\documentclass[12pt]{amsart}
\usepackage{ amsmath, amsthm, amsfonts, amssymb, color}
 \usepackage{mathrsfs}
\usepackage{amsfonts, amsmath}
 \usepackage{amsmath,amstext,amsthm,amssymb,amsxtra}
 \usepackage{txfonts} 
 \usepackage[colorlinks, citecolor=blue,pagebackref,hypertexnames=false]{hyperref}
 \allowdisplaybreaks
 \usepackage{pgf,tikz}
 \usepackage{pst-plot}

\begin{document}
\baselineskip 16pt

\newcommand\RR{\mathbb{R}}
\def\RN {\mathbb{R}^n}
\newcommand{\norm}[1]{\left\Vert#1\right\Vert}
\newcommand{\abs}[1]{\left\vert#1\right\vert}
\newcommand{\set}[1]{\left\{#1\right\}}
\newcommand{\Real}{\mathbb{R}}
\newcommand{\R}{\mathbb{R}}
\newcommand{\supp}{\operatorname{supp}}
\newcommand{\card}{\operatorname{card}}
\renewcommand{\L}{\mathcal{L}}
\renewcommand{\P}{\mathcal{P}}
\newcommand{\T}{\mathcal{T}}
\newcommand{\A}{\mathbb{A}}
\newcommand{\K}{\mathcal{K}}
\renewcommand{\S}{\mathcal{S}}
\newcommand{\Id}{\operatorname{I}}
\newcommand\wrt{\,{\rm d}}
\newcommand\Ad{\,{\rm Ad}}

\def\SL{\sqrt[m]L}
\newcommand{\mar}[1]{{\marginpar{\sffamily{\scriptsize
        #1}}}}
\newcommand{\li}[1]{{\mar{LY:#1}}}
\newcommand{\el}[1]{{\mar{EM:#1}}}
\newcommand{\as}[1]{{\mar{AS:#1}}}
 \newcommand{\comment}[1]{\vskip.3cm
\fbox{%
\color{red}
\parbox{0.93\linewidth}{\footnotesize #1}}}

\newcommand\CC{\mathbb{C}}
\newcommand\NN{\mathbb{N}}
\newcommand\ZZ{\mathbb{Z}}

\renewcommand\Re{\operatorname{Re}}
\renewcommand\Im{\operatorname{Im}}

\newcommand{\mc}{\mathcal}
\newcommand\D{\mathcal{D}}

\newtheorem{thm}{Theorem}[section]
\newtheorem{prop}[thm]{Proposition}
\newtheorem{cor}[thm]{Corollary}
\newtheorem{lem}[thm]{Lemma}
\newtheorem{lemma}[thm]{Lemma}
\newtheorem{exams}[thm]{Examples}
\theoremstyle{definition}
\newtheorem{defn}[thm]{Definition}
\newtheorem{rem}[thm]{Remark}

\numberwithin{equation}{section}
\newcommand\bchi{{\chi}}

\title[Sharp endpoint $L^p$ estimates  for   Schr\"odinger  groups ]
{Sharp endpoint $L^p$ estimates  for   Schr\"odinger  groups  }

 \author[P. Chen, X.T. Duong,   J. Li  and L. Yan]{ Peng Chen, \, Xuan Thinh Duong,  \,  Ji Li \,  and \, Lixin Yan}

 \address{Peng Chen, Department of Mathematics, Sun Yat-sen
University, Guangzhou, 510275, P.R. China}
\email{chenpeng3@mail.sysu.edu.cn}

 \address{Xuan Thinh Duong, Department of Mathematics, Macquarie University, NSW 2109, Australia}
\email{xuan.duong@mq.edu.au}

\address{Ji Li, Department of Mathematics, Macquarie University, NSW, 2109, Australia}
\email{ji.li@mq.edu.au}

\address{Lixin Yan, Department of Mathematics, Sun Yat-sen (Zhongshan) University, Guangzhou, 510275, P.R. China}
\email{mcsylx@mail.sysu.edu.cn}

  \date{\today}

\begin{abstract}
Let $L$ be a non-negative self-adjoint operator acting on $L^2(X)$
where $X$ is a space of homogeneous type with a dimension $n$. Suppose that
the heat operator $e^{-tL}$
satisfies  the generalized   Gaussian $(p_0, p'_0)$-estimates of order $m$ for some $1\leq p_0 < 2$.  In this paper
 we prove {\it sharp} endpoint $L^p$-Sobolev bound  for the Schr\"odinger group $e^{itL}$,
  that is  for every  $p\in (p_0, p'_0)$ there exists a  constant $C=C(n,p)>0$ independent of $t$ such that
 \begin{eqnarray*}
 \left\| (I+L)^{-{s}}e^{itL} f\right\|_{p} \leq  C(1+|t|)^{s}\|f\|_{p}, \ \ \ t\in{\mathbb R}, \ \ \ s\geq n\big|{1\over  2}-{1\over  p}\big|.
\end{eqnarray*}

As a consequence, the above  estimate holds for all $1<p<\infty$ when the heat kernel of $L$ satisfies a  Gaussian upper bound.
  This    extends  classical  results due to Feffermann and Stein, and Miyachi for the Laplacian
on the Euclidean spaces ${\mathbb R}^n$. We also give an application to   obtain an endpoint estimate for $L^p$-boundedness of
the Riesz means of the solutions of the Schr\"odinger equations.
 \end{abstract}

\maketitle


\section{Introduction }
\setcounter{equation}{0}

\medskip

\noindent
 {\bf 1.1. \, Background.}\ \
Consider the Laplace operator $\Delta=-\sum_{i=1}^n\partial_{x_i}^2$ on the Euclidean space $\mathbb R^n$
and  the Schr\"odinger equation
 \begin{eqnarray*}\label{e1.00}
\left\{
\begin{array}{ll}
  i{\partial_t u } +\Delta u=0,\\[4pt]
 u|_{t=0}=f
\end{array}
\right.
\end{eqnarray*}
with initial data $f$. Its solution     can be written as
 \begin{eqnarray*}\label{e1.0}
 u(x, t)=e^{it\Delta} f(x)={1\over (2\pi)^n}\int_{{\mathbb R}^n} {\widehat f}(\xi) e^{i(  \langle x, \,  \xi\rangle +t|\xi|^2 )} d\xi
 \end{eqnarray*}
 where ${\widehat f}$ denotes the Fourier transform of $f$.
It is well-known that the operator $e^{it\Delta}$ acts boundedly on $L^p({\mathbb R}^n)$ only if $p=2$;
 see  H\"ormander \cite{H1}. For $p\not= 2, $
it was shown (see for example,  \cite{Br, La, Sj}) that for $s > n|{1/ 2}-{1/p}|$, the operator
 $e^{it\Delta}$ maps the Sobolev space $L^p_{2s}(\RN)$ into $L^p(\RN)$. Equivalently, this means that
$(I+\Delta)^{-s } e^{it\Delta}$ is bounded on $L^p(\RN)$, and   this is not the case if  $0<s<  n|{1/ 2}-{1/p}|$.
 The sharp endpoint  $L^p$-Sobolev estimate
 is due to   Miyachi (\cite{Mi1, Mi}), which  states that for
every $p\in (1, \infty)$,
\begin{eqnarray}\label{e1.1}
 \left\|  (1+\Delta)^{-s} e^{it\Delta} f\right\|_{L^p(\mathbb R^n)} \leq C  (1+|t|)^{s}\|f\|_{L^p (\mathbb R^n)},
  \ \ \ t\in{\mathbb R}, \ \ \ s= n\big|{1\over  2}-{1\over  p}\big|
\end{eqnarray}
for some positive constant $C=C(n,p)$ independent of $t$. The estimate \eqref{e1.1} is sharp in another way: the factor
$(1+|t|)^{s}$ can not be improved (see \cite[p. 169-170]{Mi1}).  See also Feffermann and Stein's work \cite{FS}.
These results and their generalizations were in fact results on multipliers and relied heavily on Fourier analysis.
See, for example,  Ouhabaz's monograph \cite[Chapter  7]{O}   for
 historical background and more study on the Schr\"odinger groups.

\smallskip

The purpose of this paper is  to establish such sharp endpoint $L^p$ estimate  \eqref{e1.1} for the operators $\big(e^{itL}\big)_{t\in {\mathbb R}}$
 for a large class of non-negative self-adjoint operators acting on $L^2(X)$ on a metric  measure space $X.$
Such an operator $L$ admits a  spectral resolution
\begin{eqnarray}\label{e01.1}
Lf=\int_0^{\infty}\lambda dE_L(\lambda) f, \ \ \ \ f\in L^2(X),
\end{eqnarray}
where  $E_L(\lambda)$ is the projection-valued measure supported on the spectrum of $L$.
The operator   $   e^{itL}$ is defined by
 \begin{equation}\label{e1.3}
 e^{itL}f =   \int_0^{\infty}    e^{it\lambda}dE_L(\lambda) f
   \end{equation}
 for $f\in L^2(X)$, and forms   the Schr\"odinger group.
   By the spectral theorem (\cite{Mc}),  the operator   $  e^{itL}$  is  continuous on $L^2(X)$.
It is interesting to  investigate $L^p$-mapping properties  for
 the Schr\"odinger group $ e^{itL}$   on $L^p(X)$ for some $p, 1\leq p\leq \infty.$

 As an application of our  sharp endpoint $L^p$ estimate  for   the Schr\"odinger group $ e^{itL}$,
we also aim to obtain an endpoint estimate for $L^p$-boundedness of the Riesz means  of the solutions of the
Schr\"odinger equations.


\medskip

\noindent
 {\bf 1.2. \, Assumptions and main results.}\ \
Throughout the paper we assume that  $X$ is a metric space, with distance function $d$,
and   $\mu$ is a nonnegative, Borel doubling measure on $X$.
We say that $(X, d, \mu)$ satisfies
 the doubling property (see Chapter 3, \cite{CW})
if there  exists a constant $C>0$ such that
\begin{eqnarray}\label{e2.1}
V(x,2r)\leq C V(x, r)\quad \forall\,r>0,\,x\in X.
\end{eqnarray}
 Note that the doubling property  implies the following
strong homogeneity property,
\begin{equation}
V(x, \lambda r)\leq C\lambda^n V(x,r)
\label{e2.2}
\end{equation}
 for some $C, n>0$ uniformly for all $\lambda\geq 1$ and $x\in X$.
In Euclidean space with Lebesgue measure, the parameter $n$ corresponds to
the dimension of the space.
 There also exist $c$ and
$D, 0\leq D\leq n$ such that
\begin{equation}
V(y,r)\leq c\bigg( 1+{d(x,y)\over r}\bigg )^D V(x,r)
\label{e200}
\end{equation}
uniformly for all $x,y\in {  X}$ and $r>0$. Indeed, the
property (\ref{e200}) with
$D=n$ is a direct consequence of triangle inequality of the metric
$d$ and the strong homogeneity property. In the cases of Euclidean spaces
${\mathbb R}^n$ and Lie groups of polynomial growth, $D$ can be
chosen to be $0$.

Consider a non-negative self-adjoint operator $L$  and numbers
  $m\geq 2$ and $ 1\leq p_0\leq 2 $.
  We say that   the semigroup $e^{-tL}$  generated by $L$,  satisfies
 the generalized Gaussian  $(p_0, p'_0)$-estimate   of order $m$,
if there exist constants $C, c>0$ such that
\begin{equation*}
 \label{GGE}
\tag{${\rm GGE_{p_0,p'_0, m} }$}
\big\|P_{B(x, t^{1/m})} e^{-tL} P_{B(y, t^{1/m})}\big\|_{p_0\to {p'_0}}\leq
C V(x,t^{1/m})^{-({\frac{1}{ p_0}}-{1\over p'_0})} \exp\left(-c\left({d(x,y)^m \over    t }\right)^{1\over m-1}\right)
\end{equation*}
for  every $t>0$ and $x, y\in X$.

Note that condition \eqref{GGE} for the special case $p_0=1$ is equivalent to
$m$-th order   Gaussian estimates (see   for example, \cite{BK2}). This means that the
semigroup $e^{-tL}$ has integral kernels $p_t(x,y)$ satisfying the following Gaussian upper estimate:
\begin{equation*}
 \label{GE}
 \tag{${\rm GE}_m$}
|p_t(x,y)| \leq {C\over V(x,t^{1/m})} \exp\left(-c \, {  \left({d(x,y)^{m}\over    t}\right)^{1\over m-1}}\right)
\end{equation*}
for every $t>0, x, y\in X$, where $c, C$ are two positive constants and $m\geq 2.$
Such estimate  \eqref{GE} is typical for elliptic or sub-elliptic differential operators of order $m$
(see for example, \cite{A, ACDH, CCO,   D, DM, DOS, Gi, JN, JN2, O,  Si, Sj,  TSC}
and the references therein).
However, there are numbers of operators which satisfy generalized Gaussian estimates and, among them,
 there exist many for which classical Gaussian estimates \eqref{GE}  fail.
  This happens, e.g., for Schr\"odinger operators with rough
 potentials \cite{ScV}, second order elliptic operators with rough  lower order terms \cite{LSV}, or
 higher order elliptic operators with bounded measurable coefficients
 \cite{D2}. See also \cite{Bl, Blu, BK2, COSY, KU, SYY}.

Our main result is that under the generalized Gaussian   estimate  \eqref{GGE} for some
  $1\leq p_0< 2$, it is sufficient to ensure that
such estimate \eqref{e1.1} holds for the operator  $\big(e^{itL}\big)_{t\in {\mathbb R}}$
for $p\in (p_0, p'_0).$ Our result can be stated as follows.

\begin{thm}\label{th1.1}
Suppose  that $(X, d, \mu)$ is  a  space of homogeneous type  with a dimension $n$.  Suppose that $L$
satisfies the property \eqref{GGE}  for some
  $1\leq p_0< 2$.
Then for every $p\in (p_0, p'_0)$, there exists a  constant $C=C(n,p)>0$ independent of $t$ such that
\begin{eqnarray} \label{e1.5}
 \left\| (I+L)^{-s }e^{itL} f\right\|_{p} \leq C (1+|t|)^{s} \|f\|_{p}, \ \ \ t\in{\mathbb R}, \ \
  \ s\geq n\big|{1\over  2}-{1\over  p}\big|.
\end{eqnarray}

As a consequence,  this  estimate \eqref{e1.5} holds for all $1<p<\infty$ when the heat kernel of $L$ satisfies a  Gaussian upper bound
\eqref{GE}.
\end{thm}

 As a consequence of Theorem~\ref{th1.1}, we have the following result.

 \begin{cor}\label{cor3.3}
Suppose  that $(X, d, \mu)$ is  a  homogeneous space  with a dimension $n$.  Suppose that $L$
satisfies  the property  \eqref{GGE}  for some
  $1\leq p_0< 2$.
Then for every $p\in (p_0, p'_0)$ and $s\geq  n|{1/ 2}-{1/ p}|$, the mapping $t\to (I+L)^{-s} e^{itL}$ is strongly continuous on $L^p(X).$
\end{cor}

We  now apply  the result of   Theorem~\ref{th1.1} to  study the property of  the solution to  the Schr\"odinger equation
 \begin{eqnarray}\label{e3.22}
\left\{
\begin{array}{ll}
  i{\partial_t u }  +L u=0,\\[5pt]
 u(\cdot, 0)=f.
\end{array}
\right.
\end{eqnarray}
Then we have
$$
u(t, x)=e^{itL}f(x).
$$
One can see that the operator $e^{itL}$ is  bounded on $L^p$ only for $p=2$. Following  Sj\"ostrand \cite{Sj}, we
define the Riesz means
\begin{eqnarray}\label{RieszMean}
I_s(t)(L):= st^{-s}\int_0^t(t-\lambda)^{s-1}e^{-i\lambda L}d\lambda
\end{eqnarray}
  for $t>0$, and $I_s(t)(L)={\overline I}_s(-t)(L)$  for $t<0$ (see also \cite{BE, H}),
 and ask the question: For what values of $s$ the operators $I_s(t)(L)$ are bounded on $L^p(X)$?

 Then we have  the following result.

 \begin{thm}\label{th1.2}
Suppose  that $(X, d, \mu)$ is  a  space of homogeneous type  with a dimension $n$. Suppose that $L$
satisfies the property  \eqref{GGE}  for some
  $1\leq p_0< 2$.
Then for every $p\in (p_0, p'_0)$, there exists a  constant $C=C(n,p)>0$ independent of $t$ such that
 \begin{eqnarray}\label{nnnw}
 \left\| I_{s}(t)(L) f\right\|_{p}   \leq    C \|f\|_{p}, \ \ \ \ t\in {\mathbb R}\backslash\{0\},  \ \
  \ s\geq n\big|{1\over  2}-{1\over  p}\big|.
\end{eqnarray}

As a consequence, this  estimate \eqref{nnnw}  holds for all $1<p<\infty$ when the heat kernel of $L$ satisfies a  Gaussian upper bound
\eqref{GE}.
\end{thm}

It is known  that   such estimate \eqref{nnnw} holds due to Sj\"ostrand   \cite{Sj}  for
  the Laplacian $-\Delta  $ on  ${\mathbb R}^n$(\cite{Sj});
see also   Thangavelu's work  \cite{Th}  for the  harmonic oscillator   $-\Delta +|x|^2$ on ${\mathbb R}^n$.

\medskip

The proof of Theorem~\ref{th1.1} and Corollary~\ref{cor3.3} will be given in Section 3.
The proof of Theorem~\ref{th1.2}  will be given in Section 4.

\medskip

\noindent
 {\bf 1.3. \, Comments on the results and  methods of the proof.}\ \
On Lie groups with polynomial growth and manifolds with non-negative Ricci curvature, similar results as in \eqref{e1.1}
for $s> n\left|{1/ 2}-{1/ p}\right|$ have been first announced by Lohou\'e in \cite{Lo}, then Alexopoulos   obtained them in \cite{A}. There,
the method is to replace Fourier analysis by the finite propagation speed of the associated wave equation \cite{Ta}.   In the abstract setting
of operators on metric measure spaces,   Carron, Coulhon and Ouhabaz \cite{CCO}  showed  $L^p$-boundedness of suitable
regularizations of the Schr\"odinger group $e^{itL}$ provided $L$ satisfies Gaussian estimate \eqref{GE}.
They proposed a different approach  to use   some techniques introduced by Davies \cite{D}:
the Gaussian semigroup estimates can be extended from real times $t>0$ to complex times $z\in {\mathbb C^+}=\{z\in {\mathbb C}:\,
{\rm Re\, z}>0 \} $ such that
\begin{eqnarray}\label{ek1}
\|e^{-zL}\|_{p\to p} \leq C \left({|z|\over {\rm Re}\, z }\right)^{n|{1\over 2} -{1\over p}|+\epsilon}, \ \ \ \ \forall z\in{\mathbb C^+}.
\end{eqnarray}
On the other hand, for every  $f\in L^2\cap L^p$ and $s\geq 0,$
\begin{eqnarray*}
(I+L)^{-s} e^{itL}f={1\over \Gamma(s)}\int_0^{\infty} e^{-u}u^{s-1} e^{-(u-it)L} f du,
\end{eqnarray*}
where $\Gamma$ is the Euler Gamma function.
From \eqref{ek1}, we see that for $s> n|{1/2} -{1/p}|$,
\begin{eqnarray} \label{ek}
\|(I+L)^{-s} e^{itL}\|_{p\to p}\leq C\int_0^{\infty} e^{-u} u^{s-1}
\left( \sqrt{u^2 +t^2\over u^2}\right)^{n|{1\over 2} -{1\over p}|+\epsilon}  du
\end{eqnarray}
and so \eqref{e1.5} holds for $s> n|{1/2} -{1/p}|$.
 The Gaussian bound  \eqref{GE} assumption  on $L$ was further weakened to  the
  generalized Gaussian   estimates  \eqref{GGE}   by Blunck \cite[Theorem 1.1]{Bl} where  the estimate  \eqref{ek1} was improved
  to get   $\epsilon=0$, i.e.
 \begin{eqnarray}\label{ekee1}
\|e^{-zL}\|_{p\to p} \leq C \left({|z|\over {\rm Re}\, z}\right)^{n|{1\over 2} -{1\over p}|}, \ \ \ \ \forall z\in{\mathbb C^+}
\end{eqnarray}
for all $p\in [p_0, p'_0]$ with $p\not=\infty,$
and so \eqref{e1.5} holds for $s> n|{1/2} -{1/p}|$. However, it is direct to see that  the integral in \eqref{ek}
is {\bf  $\infty$}   when $s=n|{1/2} -{1/p}|.$

 It was an open question whether  estimate \eqref{e1.5} holds with  $s =  n|{1/2}-{1/p}|$. Based on estimate \eqref{ekee1},
 it is straightforward to obtain sharp $L^p$ frequency truncated estimates for $e^{itL} $ that
 for every $p\in (p_0, p'_0)$ and $ k\in {\mathbb Z}^+$,
 \begin{eqnarray}\label{e11.3  in}
\|e^{itL} \phi(2^{-k}L)f\|_{p}\leq C (1+2^k |t|)^{s }\|f\|_{p}, \ \ t\in{\mathbb R},
 \ \ \   s= n\big|{1\over 2}-{1\over p}\big|
 \end{eqnarray}
uniformly for $\phi$ in   bounded subsets of $C_0^{\infty}(\mathbb R) $, by writing
$$
e^{itL} \phi(2^{-k}L)f = e^{-(2^{-k}-it)L} [\phi_e(2^{-k}L) ] (f)
$$
where $\phi_e(\lambda)=e^{\lambda} \phi(\lambda)$ and then applying \eqref{ekee1}
to $e^{-(2^{-k}-it)L}$ and \cite[Theorem 1.1]{Blu} to $\phi_e(2^{-k}L)$, respectively
(for more details, see Proposition~\ref{prop1.1} below).
As a consequence of \eqref{e11.3  in},  it follows  by a standard scaling argument (\cite[p. 193]{JN})  that
for every $p\in (p_0, p'_0)$ and for  every $\epsilon>0$,
 \begin{eqnarray}\label{e1.4}
\|(I+L)^{-s -\epsilon} e^{itL}  f\|_{p}\leq C (1+  |t| )^{s }\|f\|_{p},\ \ \ \
t\in {\mathbb R}, \ \ \ s= n\big|{1\over  2}-{1\over  p}\big|.
 \end{eqnarray}

We would like to  mention that in \cite{DN},
D'Ancona and Nicola  used   a commutator argument and
 a reduction to amalgam spaces and followed  the methods of Jensen-Nakamura \cite{JN, JN2} to obtain
 estimates \eqref{e11.3  in} and \eqref{e1.4} for the Schr\"odinger group $e^{itL}$   for $p\in [p_0, p'_0]$
 in the Euclidean spaces ${\mathbb R}^n$.
 However, as in \cite[p.1021]{DN},
the authors  remarked that  ``{\it Another interesting issue is the validity of \eqref{e1.4} with $\epsilon=0$. Indeed, for $L=-\Delta$ in ${\mathbb R}^n$ and
 $1<p<\infty$, the estimate \eqref{e1.4} was proved with $\epsilon=0$ (and $t=1$) in \cite{Mi}, but this sharp form seems out of reach
 in the present generality, even for fixed $t$"}.
 Under an additional condition which is the operator $e^{itL}$ being bounded in suitable modulation spaces (see \cite[Section 5]{DN}
  for the definition),  it was proved in  \cite{DN} that estimate    \eqref{e1.4} holds with $\epsilon=0$ in the setting of ${\mathbb R}^n.$
 See also previous related results \cite{BDN,   JN, JN2}.

 \medskip

Our main result, Theorem~\ref{th1.1},  gives the sharp endpoint  estimate  \eqref{e1.4}
for the  Schr\"odinger group $e^{itL}$
     with $\epsilon=0$, namely with the optimal number of derivatives and the optimal time growth for
	 the factor
$(1+|t|)^{s}$ in \eqref{e1.4}.
The proof of    Theorem~\ref{th1.1}
  is different from those of Fefferman and Stein \cite{FS} and Miyachi \cite{Mi1, Mi}
where the results rely  heavily on Fourier analysis. In our setting,   we do not have  Fourier transform  at our disposal.
We also do not assume that the heat kernel $p_t(x,y)$
	  satisfies the standard regularity condition, thus standard techniques of
Calder\'on--Zygmund theory (\cite{St2}) are not applicable.  The
lack  of smoothness of the kernel
will be overcome in Proposition~\ref{prop3.3} below by using
some off-diagonal estimates  on heat semigroup  of non-negative self-adjoint operators,
and some techniques in the theory of
 singular integrals with rough kernels, which
  lies beyond the scope of the standard Calder\'on-Zygmund
theory (see for example,
    \cite{ACDH,  Blu, BK2, COSY,  DM, DOS, DY, KU, O, SYY} and the references therein).
 More specifically, by duality   we are reduced to prove
the estimate for $2<p<p'_0$, which will follow  by  the Littlewood-Paley inequality and a variant of
the Fefferman-Stein sharp function (see \cite{ACDH, DY, FS, Ma, SYY}),
\begin{eqnarray}\label{e1.6}
\|e^{itL}f\|_p  \leq
C \| T_\varphi f \|_p
\leq  C\|  {\mathfrak  M}_2\big( | T_\varphi f| \big)  \|_p
\leq C_p \big(  \|{\mathfrak  M}^{\#}_{T_\varphi, L, K}f  \|_p +\|f\|_p\big),
\end{eqnarray}
where
\begin{eqnarray}\label{e1.7}
T_\varphi f(x)= \left(\sum_{k\geq 0 } |\varphi_k({L})e^{itL}f(x)|^2\right)^{1/2}
\end{eqnarray}
for some cut-off function $\varphi\in C_0^{\infty}([1/2, 2])$, where $\varphi_k(\lambda)=\varphi(2^{-k}\lambda), k\geq 1$
and $\varphi_0(\lambda) +\sum_{k\geq 1} \varphi_k(\lambda)\equiv 1$ for $\lambda>0$, and  for a large $K\in{\mathbb N},$
\begin{eqnarray}\label{e1.8}
{\mathfrak  M}^{\#}_{T_\varphi, L, K}f(x)&=& \sup_{B\ni x} \left( \fint_{B} \big|T(I-e^{-r_B^mL})^K f(y)\big|^2 d\mu(y)\right)^{1/2}.
\end{eqnarray}
We then  use a variant of an argument in  \cite{LRS, RS} to decompose the function ${\mathfrak  M}^{\#}_{T_\varphi, L, K}f$
into several components  so that we can employ the off-diagonal estimates \eqref{e1.10} below. Then we
show that  the function ${\mathfrak  M}^{\#}_{T_\varphi, L, K}f$ is  in $L^p$
by using   estimate \eqref{e11.3  in} for the Schr\"odinger
group $e^{itL}$.  We note that    in the case  that
  $L$ is the Laplace operator $\Delta$ on ${\mathbb R}^n$,    the  kernel estimate  relies heavily on Fourier analysis
 since the operator $ e^{it\Delta}{  \varphi }(2^{-k} {\Delta}) $
  has the convolution kernel
  $$
  K_{e^{it\Delta}{  \varphi }(2^{-k} {\Delta})  }(x)
  ={2^{kn/2}\over (2\pi)^n} \int_{\mathbb R^n} \varphi(|\xi|^2) e^{i (2^{k/ 2}\langle x,\,  \xi\rangle +2^{k}t\, |\xi|^2 )} d\xi,
  $$
 one then uses  integration by parts to obtain
  that    for every $M>0,$
\begin{eqnarray} \label{e1.9}
\left| K_{e^{it\Delta}{  \varphi }(2^{-k} {\Delta})  }(x)\right| \leq C  2^{kn/2}   (1+2^{k/2}|x|)^{-M}
\end{eqnarray}
whenever $ |x|\geq   2^{k/2+4}$ and $ t\in[0, 1]$  (see   for example, \cite[page 62]{RS}).
However, when $L$ is a  general non-negative self-adjoint operator acting on the space $L^2(X)$ satisfying \eqref{GGE} with $p_0\in[1,2)$,
such estimate
\eqref{e1.9}
may  or may not hold.
 In our setting,  we need the following off-diagonal estimate of the operator
  $ e^{itL}{\varphi} ( 2^{-k}{L})$ (see Proposition~\ref{prop3.3} below): For every $M>0$, there exists a positive constant $C=C(n,m, M)$ independent
  of $t$ such that
\begin{eqnarray}\label{e1.10}
 \big\|P_{B_1}   e^{itL} {  \varphi}( 2^{-k}{L})    P_{B_2} f\big\|_2
 \leq C   \left(1+ {d(B_1, B_2)\over  2^{(m-1)k/m}(1+|t|)}\right)^{-M} \|P_{B_2} f  \|_2, \ \ \  t\in {\mathbb R}
\end{eqnarray}
for all balls $  B_1, B_2\subset X$ with radius $r_{B_1}=r_{B_2}\geq c2^{(m-1)k/m}(1+|t|)$ for some $c\geq 1/4$,
and $d(B_1, B_2)\geq 6 r_{B_1}$.
This new estimate is   crucial for the proof of Theorem~\ref{th1.1}.

\smallskip

The paper is organized as follows.
In Section 2 we provide some preliminary results on off-diagonal estimates of the operator $e^{itL}{\varphi} ( 2^{-k}{L})$
and   spectral multipliers and Littlewood-Paley theory,  which we need later,
mainly to prove \eqref{e1.10} in Proposition~\ref{prop3.3}.
The  proof of Theorem~\ref{th1.1}  will be given in Section 3. In Section 4 we will apply Theorem~\ref{th1.1}
to obtain $L^p$-boundedness of the Riesz means of the solution to the Schr\"odinger equation.

 \medskip

\noindent
{\bf List of notations.}
\\
$ \bullet$ $(X,d,\mu) $ denotes  a metric measure space
 with a distance $d$ and a measure  $\mu$.
 \\
$ \bullet$ $L$
is a non-negative self-adjoint operator acting on the
space $L^2(X).$
\\
$ \bullet$ For $x\in X$ and $r>0$, $B(x,r)=\{y\in X: d(x,y)<r\}$  and {$V(x,r)= \mu\big( B(x,r)\big).$
\\
$\bullet$ For   $B=B(x_B, r_B)$,    $A(x_B, r_B, 0)=B$ and
 $
A(x_B, r_B, j)=B(x_B, (j+1)r_B)\backslash B(x_B, jr_B) $ for $ j=1, 2, \ldots.
 $\\
$ \bullet$  $\delta_RF$ is defined by
$\delta_RF(x)=F(Rx)$ for $R>0$ and Borel function $F$ supported  on $ [-R, R].$
 \\
 $\bullet$
 $\lfloor t\rfloor$ denotes the integer part of $t$
for any positive real number $t$.
\\
 $\bullet$ $\NN$ is the set of positive integers.
\\
 $\bullet$ For $p\in [1,\infty]$, $p'={p}/{(p-1)}$.
 \\
 $\bullet$  For $1\le p\le\infty$ and  $f\in L^p(X,{\rm d}\mu)$,  $\|f\|_p=\|f\|_{L^p(X,{\rm d}\mu)}.$
 \\
 $\bullet$
  $\langle \cdot,\cdot \rangle$ denotes
the scalar product of $L^2(X, {\rm d}\mu)$.
\\
 $\bullet$ For $1\le p, \, q\le+\infty$,  $\|T\|_{p\to q} $ denotes the  operator norm of $T$
 from $ L^p(X, {\rm d}\mu)$ to $L^q(X, {\rm d}\mu)$.
 \\
 $\bullet$ If $T$ is  given by $Tf(x)=\int K(x,y) f(y) d\mu(y)$,  we denote by  $K_T$ the kernel of $T$.
 \\
 $\bullet$
 Given a  subset $E\subseteq X$,    $\bchi_E$  denotes  the characteristic
function of   $E$  and
 $
P_Ef(x)=\chi_E(x) f(x).
 $
\\
 $\bullet$ For every $B\subset X$,  we write $\fint_B f d\mu(y)=\mu(B)^{-1}\int_B f(y)d\mu(y)$.
\\
 $\bullet$
 For  $1\leq r <\infty$, $\mathfrak M_r$ denotes the uncentered  $r$-th maximal operator over balls in $X$, that is
 \begin{equation*}
\mathfrak  M_rf(x)=\sup_{B\ni x} \left( \fint_{B}
|f(y)|^rd\mu(y)\right)^{1/r}.
\end{equation*}
 For simplicity we denote by $\mathfrak M$ the Hardy-Littlewood maximal function $\mathfrak M_1$.

\bigskip

\section{Off-diagonal estimates and spectral multipliers}
\setcounter{equation}{0}

In this section we assume  that $(X, d, \mu)$ is a space of homogeneous type with a dimension $n$ in \eqref{e2.2} and that $L$
  is a self-adjoint non-negative operator in $L^2(X)$   satisfying the generalized Gaussian estimate
   \eqref{GGE}  for some
  $1\leq p_0< 2$.

\medskip

\noindent
{\bf 2.1. Off-diagonal estimates}. \ We start by  collecting
 some properties of the generalized Gaussian estimates obtained by Blunck and Kunstmann,
see for example, \cite{Bl, Blu,  BK2, KU} and the references therein.
For every $j\geq 1,$
we recall that  $ A(x_B, r_B, j)=B(x_B, (j+1)r_B)\backslash B(x_B, jr_B)$.
The following result originally  stated  in  \cite[Lemma 2.5]{KU} (see also \cite[Theorem 2.1]{Bl})
shows that generalized Gaussian estimates can be extended from
real times $t>0$ to complex times   $z\in {\mathbb C}$ with ${\rm Re} z>0$.
Recall that 
  $\bchi_E$  denotes the characteristic function of $E\subseteq X$ and  set $P_Ef(x)=\chi_E(x) f(x)$.

\begin{lemma}\label{le2.1}
  Let $m\geq 2$ and $ 1\leq p\leq 2\leq q\leq \infty$, and $L$ be a
 non-negative self-adjoint operator   on $L^2(X)$. Assume that  there exist constants $C, c>0$ such that
for all $t>0$,  and all $x,y\in X,$
\begin{eqnarray*}
\big\|P_{B(x, t^{1/m})} e^{-tL} P_{B(y, t^{1/m})}\big\|_{p\to {q}}\leq
 C V(x,t^{1/m})^{-({1\over p}-{1\over q})} \exp\Big(-c\Big({d(x,y) \over    t^{1/m}}\Big)^{m\over m-1}\Big).
\end{eqnarray*}
Let $r_z=({\rm Re}\, z)^{{1\over m}-1 } |z|$
for each  $z\in{\mathbb C}$ with ${\rm Re}\, z >0$.
\begin{itemize}

\item[(i)]  There exist two positive constants $C' $ and $ c'$ such that for all $r>0, x\in X, $
 and $z\in{\mathbb C}$ with ${\rm Re}\, z >0,$
\begin{eqnarray*}
 &&\hspace{-1.5cm}\big\|P_{B(x, r)}    e^{-zL} P_{B(y, r)}\big\|_{p \to q}\\
 &\leq &C' V(x,r)^{-({1\over p }-{1\over q})} \Big(1+{r\over r_z}\Big)^{n({1\over p }-{1\over q})}
  \Big({|z|\over {\rm Re}\,  z}\Big)^{n({1\over p }-{1\over q})}
  \exp\Big(-c' \Big({ d(x,y)   \over   r_z}  \Big)^{m\over m-1} \Big).
\end{eqnarray*}

\item[(ii)]  There exist two positive constants $C''$ and $ c''$ such that for all $r>0, x\in X, k\in{\mathbb N}$
 and $z\in{\mathbb C}$ with ${\rm Re}\, z >0,$
\begin{eqnarray*}
 &&\hspace{-1.5cm}\big\|P_{B(x, r)}    e^{-zL} P_{A(x, r, k)}\big\|_{p \to q}\\
 &\leq &C'' V(x,r)^{-({1\over p }-{1\over q})} \Big(1+{r\over r_z}\Big)^{n({1\over p}-{1\over q})}
  \Big({|z|\over {\rm Re} \, z}\Big)^{n({1\over p }-{1\over q})} k^n
  \exp\Big(-c'' \Big({ r   \over    r_z} k\Big)^{m\over m-1} \Big).
\end{eqnarray*}
\end{itemize}
 \end{lemma}
 \begin{proof}
 For the
 detailed proof
 we refer  readers to \cite{KU}.
  Here we only    mention that the proof of Lemma \ref{le2.1} relies on the Phragm\'en-Lindel\"of theorem.
 \end{proof}

 \medskip
  Next suppose that $m\geq 2$.
  We say that the semigroup $e^{-tL}$ generated by non-negative self-adjoint operator $L$ satisfies
 {\it $m$-th order Davies-Gaffney  estimates},
if there exist constants $C, c>0$ such that
for all $t>0$,  and all $x,y\in X,$
 \begin{equation*}
 \label{DG}
 \tag{${\rm DG}_m$}
\big\|P_{B(x, t^{1/m})} e^{-tL} P_{B(y, t^{1/m})}\big\|_{2\to {2}}\leq
C   \exp\left(-c\left({d(x,y) \over    t^{1/m}}\right)^{m\over m-1}\right).
\end{equation*}
Note that if  condition \eqref{GGE} holds for
 some $1\leq p_0\leq 2 $ with $p_0<2$,
then  the semigroup
$e^{-tL}$  satisfies estimate \eqref{DG}.

The following Lemma  describes  a useful consequence of $m$-order
Davies-Gaffney estimates (see \cite[Lemma 2.2]{SYY}).

\begin{lemma}\label{le2.2}
Let $m\geq 2$ and  $L$ satisfies the   Davies-Gaffney  estimates \eqref{DG}.
Then for every  $M>0$, there exists a constant $C=C(M)$ such that
for every $j=2,3,\ldots$
\begin{eqnarray}\label{e2.3} \hspace{1cm}
 \big\|P_{B}    F({L}) P_{A(x_B, r_B, j)}\big\|_{2\to 2}\leq
 C j^{-M}  \big(\sqrt[m]{R} r_B)^{-(M+n)}   \|\delta_R F\|_{W^{M+n+1}_2}
\end{eqnarray}
for all  balls $B\subseteq  X$, and all Borel functions $F$  such that supp $F\subseteq [-R, R]$.
 \end{lemma}

\begin{proof}  Let $G(\lambda)=(\delta_RF)({\lambda}) e^{\lambda}.$ In virtue of the Fourier inversion formula
$$
G(L/R)e^{-L/R}={1\over 2\pi} \int_{\mathbb R} e^{(i\tau-1)R^{-1}L} {\hat G}(\tau)d\tau
$$
so
$$
 \|P_{B}    F({L}) P_{A(x_B, r_B, j)}\|_{2\to 2} \leq
{1\over 2\pi} \int_{\R} |{\hat G}  (\tau)| \,   \|P_{B}   e^{(i\tau-1)R^{-1}L} P_{A(x_B, r_B, j)} \|_{2\to 2}
  d\tau.
$$
  By (ii) of Lemma~\ref{le2.1} (with $r_z =\sqrt{1+\tau^2}/\sqrt[m]{R}$),
 \begin{eqnarray*}
\|P_{B}   e^{(i\tau-1)R^{-1}L} P_{A(x_B, r_B, j)} \|_{2\to 2}&\leq&
Cj^n
  \exp\left(-c\left({ \sqrt[m]{R} j r_B \over    \sqrt{1+\tau^2}}\right)^{m\over m-1}   \right)\\
  &\leq& C_M  j^n   \left( {  \sqrt[m]{R} jr_B \over \sqrt{1+\tau^2}    } \right)^{-M-n}\\
  &\leq& C j^{-M} \big(1+\tau^2)^{ M+n \over 2}  \big(\sqrt[m]{R} r_B)^{-(M+n)}.
\end{eqnarray*}
Therefore (compare \cite[(4.4)]{DOS})
  \begin{eqnarray*}
   &&\hspace{-1cm}\|P_{B}    F({L}) P_{A(x_B, r_B, j)}\|_{2\to 2}\\
   &\leq&
 Cj^{-M}  \big(\sqrt[m]{R} r_B)^{-(M+n)}
   \int_{\R} |{\hat G}  (\tau)|  \big(1+\tau^2)^{ M+n \over 2}  d\tau\nonumber\\
   &\leq&
 C j^{-M}  \big(\sqrt[m]{R} r_B)^{-(M+n)}
   \left(\int_{\R} |{\hat G}  (\tau)|^2  \big(1+\tau^2)^{M+n+1}   d\tau\right)^{1/2}
    \left(\int_{\R} \big(1+\tau^2)^{-1}  d\tau\right)^{1/2}\nonumber\\
  &\leq&    C  j^{-M}  \big(\sqrt[m]{R} r_B)^{-(M+n)}   \|G\|_{W^{M+n+1}_2}.
\end{eqnarray*}
 However, supp $F\subseteq [-R, R]$ and supp $\delta_R F\subseteq  [-1, 1]$ so
 $$
  \|G\|_{W^{M+n+1}_2} \leq C \|\delta_R F\|_{W^{M+n+1}_2}.
 $$
This completes the proof of Lemma \ref{le2.2}.
\end{proof}

 The proof of Theorem 1.1 relies on  the following off-diagonal estimates for
 $  e^{itL}{  \phi}_k( {L})$, where $\phi\in C_0^{\infty}([{1/4}, 4])$ is a cut-off function
 and  $\phi_k(s)=\phi(2^{-k} s)$ for every $k\geq 1$.

 \begin{prop}\label{prop3.3}
Let $m\geq 2$ and  $L$ satisfies the   Davies-Gaffney  estimates \eqref{DG}.
 For every  $M>0$,   $K\geq 1$, $s>0$, $t\in{\mathbb R}$ and $ k\geq 1$,  there exists  a  constant
$C=C(M, n, K)$ independent of   $t, s$, and $k$  such that
\begin{eqnarray}\label{e3.10}
 \big\|P_{B_1}  (I-e^{-sL} )^K e^{itL}{  \phi}_k( {L})    P_{B_2} f\big\|_2
 \leq C   \left(1+{d(B_1, B_2)\over  2^{(m-1)k/m}(1+|t|)}\right)^{-M} \|P_{B_2} f  \|_2
\end{eqnarray}
  for all $  B_i\subset X$ with  $r_{B_1}=r_{B_2}\geq c2^{(m-1)k/m}(1+|t|)$ for some $c\geq 1/4$, and $d(B_1, B_2)\geq 6 r_{B_1}$.
\end{prop}

\smallskip
To prove Proposition~\ref{prop3.3}, we need  the following Lemmas~\ref{le2.3} and ~\ref{le2.4}.

 \begin{lemma}\label{le2.3}
 Let $m\geq 2$ and  $L$ satisfies the   Davies-Gaffney  estimates \eqref{DG}. Then
 for every  $M>0,$ $  k\in {\mathbb N}^+$ and $t\in{\mathbb R}$,  there exists  a  constant
$C=C(M, m, n)$ independent of  $t$ and $k $   such that for every $j=2,3,\ldots$
\begin{eqnarray}\label{le2.3 e1}
 \big\|P_B   e^{-(2^{-k}-it)L}  P_{A(x_B, r_B, j)} f\big\|_2
 \leq C   j^{-M}  \left(1+\frac{r_B}{ 2^{(m-1)k/m}(1 + |t| )}\right)^{-M} \|P_{A(x_B, r_B, j)} f  \|_2
\end{eqnarray}
 for all balls $B\subset X$ with  $r_B\geq c2^{(m-1)k/m}(1+|t|)$ for some $c\geq 1/4$.

  As a consequence,  we have
\begin{eqnarray*}
 \big\|P_B   e^{-(2^{-k}-it)L}    P_{X\backslash 2B} f\big\|_2
 \leq C     \mu(B)^{1/2}{\mathfrak  M}_2(f)(x)
\end{eqnarray*}
for all balls $B\subset X$ with  $r_B\geq c2^{(m-1)k/m}(1+|t|)$ for some $c>1/4$ and for every $x\in B$.
\end{lemma}

\begin{proof}
Note that
 \begin{eqnarray*}
 \|P_B   e^{-(2^{-k}-it)L} P_{X\backslash 2B} f \|_2  &\le&
    \sum_{j=2}^{\infty}
 \|P_{B}   e^{-zL} P_{A(x_B, r_{B }, j)} f  \|_2
\end{eqnarray*}
with  $z=(2^{-k}-it)$. It is clear that ${\rm Re}\, z= 2^{-k}>0$, and so $r_z=({\rm Re}\, z)^{{1\over m}-1}|z|=2^{(m-1)k/m}\sqrt{|t|^2 + 2^{-2k}}.$
By (ii) of Lemma~\ref{le2.1}, we see that  for every ball $ B\subset X$ with  $r_B\geq  2^{(m-1)k/m}(1+|t|), k\geq 0$,
\begin{eqnarray} \label{e2.4}
 \big\|P_B   e^{-(2^{-k}-it)L} P_{A(x_B, r_{B }, j)}\big\|_{2 \to 2}
 &\leq &C j^n
  \exp\left(-c  \left({ r_B j \over    2^{(m-1)k/m}\sqrt{ 2^{-2k} + |t|^2 }} \right)^{m\over m-1} \right)\nonumber\\
 & \leq&  C_{M}  j^{-M+n}\left(1+\frac{r_B}{ 2^{(m-1)k/m}(1 + |t| )}\right)^{-M}
\end{eqnarray}
for every $M>0$. Hence, \eqref{le2.3 e1} holds.
This, in combination with the fact that for every $x\in B$,
\begin{eqnarray} \label{e2.5}
 \|P_{A(x_B, r_{B }, j)}f  \|_2 &\leq&  \mu((j+1)B)^{1/2} \left(\fint_{(j+1)B }|  f(y)|^2 d\mu(y)\right)^{1/2}
 \nonumber\\
 &\leq& C (j+1)^{n/2} \mu(B)^{1/2}{\mathfrak  M}_2\left(  f \right)(x),
\end{eqnarray}
 yields that
 \begin{eqnarray*}
  \|P_B   e^{-(2^{-k}-it)L} P_{X\backslash 2B} f \|_2
  &\le&
    C \sum_{j=2}^{\infty} j^{-(M-{3n\over 2})}  \mu(B)^{1/2}{\mathfrak  M}_2\left(  f \right)(x)\\
  &\le&
    C    \mu(B)^{1/2}{\mathfrak  M}_2\left(  f \right)(x)
\end{eqnarray*}
as long as we choose $M>3n/2$ in \eqref{e2.4}. This proves Lemma~\ref{le2.3}.
\end{proof}

 \begin{lemma}\label{le2.4}  Let $m\geq 2$ and  $L$ satisfies the  Davies-Gaffney  estimates \eqref{DG}.
For a given  $\phi\in C_0^{\infty}([{1\over4}, 4])$, we write  $\phi_e(\lambda)=e^{\lambda} \phi(\lambda).$
  Then
for  every   $M>0, k\in {\mathbb N}^+$ and $s>0$,  there exists  a  constant
$C=C(m, n,  M)$ independent of   $k $ and $s$ such that   for every $j=2,3,\ldots$
\begin{eqnarray*}
 \big\|P_B   (I-e^{-sL} )^K  {\phi}_e( 2^{-k}{L}) P_{A(x_B, r_B, j)} f\big\|_2
 \leq C   j^{-M} \big(2^{k/m} r_{B }\big)^{-M-n} \|P_{A(x_B, r_B, j)} f  \|_2
\end{eqnarray*}
 for all $ B\subset X$ with  $r_B\geq c2^{(m-1)k/m}$ for some $c\geq 1/4$.

  As a consequence,  we have
\begin{eqnarray*}
 \big\|P_B   (I-e^{-sL} )^K  {\phi}_e( 2^{-k}{L})   P_{X\backslash 2B} f\big\|_2
 \leq C    \mu(B)^{1/2}{\mathfrak  M}_2(f)(x).
\end{eqnarray*}
\end{lemma}

\begin{proof}
We write
 \begin{eqnarray*}
 \|P_B (I-e^{-sL})^K  {\phi}_e( 2^{-k}{L})  P_{X\backslash 2B} f \|_2 &\le&
    \sum_{j=2}^{\infty}
 \|P_{B}   (1- e^{-sL})^K {\phi}_e( 2^{-k}{L}) P_{A(x_B, r_{B }, j)} f  \|_2.
\end{eqnarray*}
Note that the function
$ (1- e^{-s\lambda})^K e^{ 2^{-k}\lambda}\phi_k(\lambda) $ is supported in $  [2^{k-2}, \ 2^{k+2}]$. We apply Lemma~\ref{le2.2}  with $R=2^{k+2}$
to obtain that for every $M>0$ and $j\geq 2$,
 \begin{eqnarray} \label{e2.6}
 \big\|P_{B }  (1- e^{-s L})^K  {\phi}_e( 2^{-k}{L}) P_{A(x_B, r_{B }, j)}\big\|_{2\to 2}
 &\leq&    C j^{-M} \big(2^{k/m} r_{B }\big)^{-M-n} \|\delta_{2^{k+2}}\big( (1- e^{-s\lambda})^Ke^{ 2^{-k}\lambda}
 \phi_k(\lambda) \big) \|_{  W^{M+n+1}_2}\nonumber\\
 &\leq&   C j^{-M} \big(2^{k/m} r_{B }\big)^{-M-n} \|  (1- e^{-2^{ (k+2)}s\lambda})^K e^{ 4\lambda}\phi(4\lambda)  \|_{  W^{M+n+1}_2}
 \nonumber\\
 &\leq&   C j^{-M}\big(2^{k/m} r_{B }\big)^{-M-n} .
\end{eqnarray}
This, in combination with  \eqref{e2.5},
yields that for every $x\in B$,
 \begin{eqnarray*}
 \|P_B (I-e^{-sL})^K  {\phi}_e( 2^{-k}{L})  P_{X\backslash 2B} f \|_2
  &\le&
    C\sum_{j=2}^{\infty}  j^{-(M-{n\over 2})} \big(2^{k/m} r_{B }\big)^{-M-n}  \mu(B)^{1/2}{\mathfrak  M}_2\left( f \right)(x)\\
  &\le&
    C  \mu(B)^{1/2}{\mathfrak  M}_2\left( f \right)(x)
\end{eqnarray*}
as long as we choose $M>n/2$ in the first inequality above and notice the fact that $2^{k/m}r_B\geq 1/4$.  This proves Lemma~\ref{le2.4}.
\end{proof}

\medskip

\begin{proof}[Proof of Proposition~\ref{prop3.3}]
Let us  show \eqref{e3.10} when  $d(B_1, B_2)\geq 6 r_{B_1}$.
 By spectral theory, we write
\begin{eqnarray*}
(I-e^{-sL} )^K e^{itL}{  \phi}_k({L}) &= &
e^{-(2^{-k}-it)L} \big[  (I-e^{-sL} )^K  {  \phi}_e(2^{-k}{L})  \big]=S_{k, t}(L) T_k(L)
\end{eqnarray*}
where we write  $\phi_e(\lambda)=e^{\lambda} \phi(\lambda)$,
$$
S_{k, t}(L)=e^{-(2^{-k}-it)L}
$$
and
$$ T_k(L)= (I-e^{-sL} )^K  {\phi}_e( 2^{-k}{L}).
$$
Set $G=\{ x: {\rm dist} (x, B_1) \leq d(B_1, B_2)/2 \}$.
Then it is clear that ${\rm dist} (\, B_2, {\bar G} )\geq d(B_1, B_2)/2$,
where we use ${\bar G}$ to denote the topological closure of the set $G$. Moreover, from the definition of $G$, it is also clear that ${\rm dist} (X\backslash G, B_1)\geq
d(B_1, B_2)/3.$
Furthermore, based on the above observations we have
$$
G \subset \bigcup_{j= \big\lfloor {d(B_1,B_2)\over 2r_{B_2}} \big\rfloor}^{ \big\lfloor 2+{d(B_1,B_2)\over  r_{B_2}}\big\rfloor +1}A(x_{B_2}, r_{B_2}, j)
\quad{\rm\ and\ }\quad X\backslash G\subset   \bigcup_{j= \big\lfloor {d(B_1,B_2)\over 2r_{B_1}} \big\rfloor -1}^\infty A(x_{B_1}, r_{B_1}, j),
$$
where $\lfloor a\rfloor$ denotes the greatest integer that is smaller than $a$.

 Then by noting that $S_{k, t}(L)$ is uniformly bounded on $L^2(X)$ and by Lemma~\ref{le2.4},
 \begin{eqnarray}\label{e3.12}
\big\|P_{B_1} S_{k, t}(L) \left( P_G T_{k}(L) P_{B_2} f  \right)\big\|_{2}
&\le&  \big \| S_{k, t}(L) \left( P_G T_{k}(L) P_{B_2} f  \right)\big\|_{2} \nonumber\\
&\le& C\big \|   P_G T_{k}(L) P_{B_2} f  \big\|_{2} \nonumber\\
&\le& C \sum_{j=  \lfloor {d(B_1,B_2)/ ( 2r_{B_2})}  \rfloor}^{  \lfloor 2+{d(B_1,B_2)/r_{B_2}} \rfloor +1}  \big \|   P_{A(x_{B_2}, r_{B_2}, j)} T_{k}(L) P_{B_2} f  \big\|_{2}\nonumber\\
  &\le& C\sum_{j=  \lfloor {d(B_1,B_2)/(2r_{B_2})}  \rfloor}^{  \lfloor 2+{d(B_1,B_2)/r_{B_2}} \rfloor +1}  j^{-M} r_{B_2}^{-M}
  \big \|    P_{B_2} f  \big\|_{2}
 \nonumber\\
&\le&
    C  \left(1+{d(B_1, B_2)\over r_{B_2}}\right)^{-M+1}r_{B_2}^{-M} \left\|P_{B_2} f  \right\|_2\nonumber\\
&\le&
    C 
    \left(1+{d(B_1, B_2)\over 2^{(m-1)k/m}(1+|t|)}\right)^{-M+1} \left\|P_{B_2} f  \right\|_2
\end{eqnarray}
for any $M>0$,  
where in the last inequality we use the facts that $r_{B_2}\geq 1/4$ and that $d(B_1, B_2)> 2^{k(m-1)}(1+|t|)$.

On the other hand,
we apply  Lemma~\ref{le2.3} and the fact that  $T_{k}(L) $ is uniformly bounded on $L^2(X)$ to see that for every $M>0$,
 \begin{eqnarray}\label{e3.13}
\big\|P_{B_1}  S_{k, t}(L) \left( P_{X\backslash G} T_{k}(L) P_{B_2} f  \right)\big\|_{2}
&\le& \sum_{j= \lfloor {d(B_1,B_2)/ (2r_{B_1})} \rfloor -1}^{ \infty}  \big\|P_{B_1}  S_{k, t}(L) P_{A(x_{B_1}, r_{B_1}, j)}\left(  T_{k}(L) P_{B_2} f  \right)\big\|_{2} \nonumber\\
&\le& \sum_{j= \lfloor {d(B_1,B_2)/ (2r_{B_1})} \rfloor -1}^{ \infty}  j^{-M}\left(1+\frac{r_{B_1}}{ 2^{(m-1)k/m}(1 + |t| )}\right)^{-M}\big\| T_{k}(L) P_{B_2} f   \big\|_{2} \nonumber\\
&\le&
    C  \left(1+{d(B_1, B_2)\over 2^{(m-1)k/m}(1+|t|)}\right)^{-M} \left\|P_{B_2} f  \right\|_2
\end{eqnarray}
 Therefore,  we combine the estimates \eqref{e3.12} and \eqref{e3.13} to obtain that for every $M>0$,
 \begin{eqnarray*}\label{e8.2}
 \big\|P_{B_1} S_{k, t}(L)   \left(T_{k}(L) P_{B_2} \right) \big\|_2
 &\leq& \big \|P_{B_1}  S_{k, t}(L) \left( P_G T_{k, t}(L) P_{B_2}  f  \right)\big\|_{2}\\
&&+\big\|P_{B_1}  S_{k, t}(L) \left( P_{X\backslash G} T_{k}(L) P_{B_2}  f  \right)\big\|_{2}\nonumber\\
   &\le&
    C  \left(1+{d(B_1, B_2)\over 2^{(m-1)k/m}(1+|t|)}\right)^{-M} \big\|P_{B_2} f \big\|_2,
 \end{eqnarray*}
which shows that   \eqref{e3.10} holds.
 The proof of Proposition~\ref{prop3.3} is complete.
\end{proof}

\medskip

In order to prove Theorem~\ref{th1.2}, we also need the following estimate for the operator $e^{itL}{  \phi}_k( {tL}), t>0$.
Recall that   $\phi\in C_0^{\infty}([{1/4}, 4])$ is  a cut-off function
 and  $\phi_k(s)=\phi(2^{-k} s)$ for every $k\geq 1$.

\begin{prop} \label{prop13.33}
  Let $m\geq 2$ and  $L$ satisfies the   Davies-Gaffney  estimates \eqref{DG}. For every  $M>0$,   $K\in{\mathbb N}^+$,
  $s>0, t>0 $ and $ k\geq 1$,  there exists  a  constant
$C=C(M, n, K)$ independent of   $t, s$, and $k$  such that
\begin{eqnarray*}
 \big\|P_{B_1}  (I-e^{-sL} )^K e^{itL}{  \phi}_k( {tL})    P_{B_2} f\big\|_2
 \leq C   \left(1+{d(B_1, B_2)\over  2^{(m-1)k/m}t^{1/m}}\right)^{-M} \|P_{B_2} f  \|_2
\end{eqnarray*}
  for all $  B_i\subset X$ with  $r_{B_1}=r_{B_2}\geq c2^{(m-1)k/m}t^{1/m}$ for some $c\geq 1/4$.
\end{prop}

\begin{proof}
 The proof of Proposition~\ref{prop13.33} can be obtained by making minor modifications with the proof of Proposition~\ref{prop3.3},
 we leave the detail to the reader.
 \end{proof}

\medskip

\noindent
{\bf 2.2. Spectral multipliers}.
The following result is a standard known result in the theory of spectral multipliers of non-negative selfadjoint
operators.

\begin{prop}\label{prop2.5}
Let $m\geq 2$. Suppose that $(X, d, \mu)$ is a space of homogeneous type with a dimension $n$.
 Suppose that $L$
satisfies the property \eqref{GGE}  for some
  $1\leq p_0< 2$. Then we have
\begin{itemize}
\item[(a)]
Assume in addition that $F$ is  an even bounded Borel function  such that
 $
\sup_{R>0}\|\eta\delta_RF\|_{C^\beta}<\infty
 $
for some integer $\beta> n/2 + 1$ and some non-trivial function $\eta\in C_0^{\infty}(0, \infty)$.
Then the operator  $F(L)$ is bounded on $L^{p}(X)$ for all $p_0<p<p_0'$,
\begin{eqnarray}\label{n}
\|F({L})\|_{p\to p}\leq C_\beta\left(\sup_{R>0}\|\eta\delta_RF\|_{C^\beta}+F(0)\right).
\end{eqnarray}

\item[(b)]
Fix a non-zero $C^{\infty}$ bump function  $\varphi$ on $\mathbb R$ such that
$
{\rm supp} \, \phi \subseteq ({1/2}, 2) $  for all $\lambda>0$
and set $\varphi_{0}(\lambda)=\sum_{\ell\leq 0}\varphi(2^{-\ell}\lambda)$ and
$\varphi_{k}(\lambda)=\varphi(2^{-k}\lambda)$ for $k=1, 2, \cdots$. Then
for all $p_0<p<p_0'$,
\begin{eqnarray}\label{nn}
\Big\|\Big(\sum_{k=0}^{\infty}\big|\varphi_k({L})f\big|^2\Big)^{1/2} \Big\|_p\leq C_p\|f\|_p.
\end{eqnarray}
In addition, if   $ \sum_{k\geq 0}\varphi_k(\lambda)=1 $ for all $\lambda>0$, then we have
\begin{eqnarray}\label{xx}
\|f\|_p\cong C_p \Big\|\Big(\sum_{k=0}^{\infty}\big|\varphi_k({L})f\big|^2\Big)^{1/2} \Big\|_p, \ \ \   p_0<p<p_0'.
\end{eqnarray}
\end{itemize}
\end{prop}

\begin{proof}
Assertion (a) follows from \cite[Theorem 1.1]{Blu}, see also  \cite[Lemma 4.5]{COSY}.  The proof of assertion  (b) follows
from Stein's classical proof \cite[Chapter IV]{St2}. We give a brief argument
of this proof for completeness and convenience for the reader.

Let us introduce the Rademacher function, which is defined as follows: i) The function $r_0(t)$ is defined by $r_0(t)=1$
on $[0, 1/2]$ and $r_0(t)=-1$ on $(1/2, 1)$, and then extended to ${\mathbb R}$ by periodicity; ii) For $k\in{\mathbb N}\backslash \{0\},
r_k(t)=r_0(2^kt).$ Define
$$
F(t, \lambda)= \sum_{k=0}^{\infty} r_k(t) \varphi_k(\lambda).
$$
A straightforward computation shows that for every integer $\beta> n/2 + 1$, $\sup_{R>0}\|\eta F(t, R\lambda)\|_{C^\beta} \leq C_\beta$
uniformly in $t\in [0,1].$ Then we apply \eqref{n} to see  that for all $p\in (p_0, p'_0),$
\begin{eqnarray*}
\|F (t, {L}) f\|_p=\Big\|\sum_{k=0}^{\infty}r_k(t) \varphi_k({L})f \Big\|_p\leq C \|f\|_p
\end{eqnarray*}
with $C>0$  uniformly in $t\in [0,1].$ This, in combination with  the standard inequality for Rademacher functions:
\begin{eqnarray*}
\Big(\sum_{k=0}^{\infty}\left|\varphi_k({L})f\right|^2\Big)^{p/2} \cong \int_0^1\big|\sum_{k=0}^{\infty} r_k(t) \varphi_k({L})f \big|^p
dt,
\end{eqnarray*}
 yields
\begin{eqnarray*}
  \Big\|\Big(\sum_{k=0}^{\infty}\big|\varphi_k({L})f\big|^2\Big)^{1/2} \Big\|_p
  \cong \left(\int_0^1\Big\|\sum_{k=0}^{\infty}r_k(t) \varphi_k({L})f \Big\|_p^pdt\right)^{1/p} \leq C_p\|f\|_p.
\end{eqnarray*}
This proves \eqref{nn}.

When  $\sum_{k\geq 0}\varphi_k(\lambda)=1 $ for all $\lambda>0,$
it follows by
 the spectral theory \cite{Mc} that  $ \sum_{k\geq 0}\varphi_k(L)f=f $ for   every  $f\in L^2.$
From it,   we obtain \eqref{xx}  by using \eqref{nn} and    the standard duality argument (see for example, \cite[Chapter IV]{St2}).
This completes the proof of Proposition \ref{prop2.5}.
\end{proof}

\medskip

\section{Sharp endpoint $L^p$-Sobolev estimates for    Scr\"odinger groups}\label{sec3}
\setcounter{equation}{0}

In this section we  prove  \eqref{e1.5}   in Theorem~\ref{th1.1}.  First, we note that from \eqref{ek},
 estimate \eqref{e1.5}  holds for $s>n|{1/2}-{1/p}|.$
By  duality, it suffices to verify
\eqref{e1.5}     for $ 2\leq  p<p'_0$  and $s=n|{1/2}-{1/p}|.$
Also,    it follows  by the spectral theory \cite{Mc} that \eqref{e1.5} holds  for $p=2$.
For $p\not=2$, we recall  that when $L$ satisfies  the
  generalized Gaussian   estimates  \eqref{GGE}  for some
  $1\leq p_0< 2$, it was proved by Blunck \cite[Theorem 1.1]{Bl} that for every $z\in{\mathbb C^+},$
 \begin{eqnarray}\label{complex}
\|e^{-zL}\|_{p\to p} \leq C \left({|z|\over {\rm Re}\, z}\right)^{n|{1\over 2} -{1\over p}|}
\end{eqnarray}
for all $p\in [p_0, p'_0]$ with $p\not=\infty.$
From this, we   have the following sharp $L^p$ frequency truncated estimates for the Schr\"odinger group.

\begin{prop}\label{prop1.1}
Suppose  that $(X, d, \mu)$ is a space of homogeneous type  with a  dimension $n$. Suppose that $L$
satisfies the property \eqref{GGE}  for some
  $1\leq p_0< 2$.
Then for every $p\in (p_0, p'_0)$ and $ k\geq 0$,
 \begin{eqnarray}\label{e11.3}
\|e^{itL} \phi(2^{-k}L)f\|_{p}\leq C (1+2^k |t|)^{s }\|f\|_{p}, \ \ t\in{\mathbb R},
 \ \ \   s= n\big|{1\over 2}-{1\over p}\big|
 \end{eqnarray}
uniformly for $t\in{\mathbb R}$ and for $\phi$ in   bounded subsets of $C_0^{\infty}(\mathbb R).$
\end{prop}

\begin{proof} To show \eqref{e11.3}, we apply   \eqref{complex} with $z=2^{-k}-it$
to get  that for every $\phi\in C_0^{\infty}({\mathbb R})$,
\begin{eqnarray*}
\|e^{itL}\phi(2^{-k}L)\|_{p\to p}&=&  \left\|e^{-(2^{-k}-it)L} \big[  \phi_e(2^{-k}L)\big]\right\|_{p\to p}
\leq  C   (1+2^k |t|)^{s } \| \phi_e(2^{-k}L)\|_{p\to p} \nonumber\\
 &\leq&  C   (1+2^k |t|)^{s },
\end{eqnarray*}
where $\phi_e(\lambda)=e^{\lambda} \phi(\lambda)$. In the last inequality we used
Proposition~\ref{prop2.5} to know that the operator $\phi_e(2^{-k}L)$ is bounded on $L^p(X)$
all $p\in (p_0, p'_0).$
This completes the proof of Proposition~\ref{prop1.1}.
\end{proof}

  \medskip

To prove Theorem~\ref{th1.1},  let us
  introduce some tools needed in the proof. Let  $T$ be  a sublinear operator which is bounded on $L^{2}(X)$
and  $\{A_r\}_{r>0}$ be a family of linear  operators acting on  $L^{2}(X)$.    For $f\in L^2(X)$,
we follow  \cite{ACDH} to  define
$$
{\mathfrak  M}^{\#}_{T, A}f(x)= \sup_{B\ni x} \left( \fint_{B} \big|T(I-A_{r_B}) f\big|^2 d\mu\right)^{1/2},
$$
where the supremum is taken over all balls $B$ in $X$ containing $x$, and $r_B$ is the radius of $B.$
 Then we have the following result. For its  proof, we refer readers  to  \cite[Lemma 2.3]{ACDH}, \cite[Lemma 5.4]{DY}
 and \cite[Proposition 3.2]{SYY}.

\smallskip

\begin{prop}\label{prop3.2}
Suppose that   $T$ is  a sublinear operator which is bounded on $L^{2}(X)$ and that $q\in (2, \infty].$
Assume that $\{A_r\}_{r>0}$ is  a family of linear  operators acting on  $L^{2}(X)$ and that
\begin{eqnarray}\label{e3.1}
\left(  \fint_B|T A_{r_B}f(y)|^{q} d\mu(y)\right)^{1/q}
\leq C  {\mathfrak  M}_2 \big( Tf  \big) (x)
\end{eqnarray}

\noindent
for all $f \in L^{2}(X) $, all $x\in X$ and all balls $B\ni x$, $r_B$ being
the radius  of $B$.

Then for $0<p<q$, there exists $C_p$ such that
\begin{eqnarray}\label{e3.2}
 \left\|  {\mathfrak  M}_2\big(   Tf  \big) \right\|_p\leq C_p \left(  \left\|{\mathfrak  M}^{\#}_{T, A}f \right\|_p +\|f\|_p\right)
\end{eqnarray}
 for every $f\in L^2(X)$ for which the left-hand side is finite (if $\mu(X)=\infty$, the term $C_p\|f\|_p$
 can be omitted in the right-hand side of \eqref{e3.2}).
\end{prop}

 \bigskip

 \noindent
{\em Proof of}   { Theorem~\ref{th1.1}.} \  Let us show Theorem~\ref{th1.1} for $2<p<p'_0$
and $s=n|{1/2}-{1/ p}|$.
 We fix a non-zero $C^{\infty}$ bump function  $\varphi$ on $\mathbb R$ such that
\begin{eqnarray}\label{e3.3}
{\rm supp} \, \varphi \subseteq ({1\over 2}, 2) \ \ {\rm and} \ \
\sum_{\ell\in {\mathbb Z}}\varphi(2^{-\ell}\lambda)=1 \ \ \ {\rm for\ all}\  \lambda>0
\end{eqnarray}
and set $\varphi_{0}(\lambda)=\sum_{\ell\leq 0}\varphi(\lambda/2^{\ell})$ and
$\varphi_{\ell}(\lambda)=\varphi(\lambda/2^{\ell})$ for $\ell=1, 2, \ldots$.

For this fixed bump function $\varphi$,  we consider an operator $T_\varphi$, given by
\begin{align}\label{Tphi}
T_\varphi f(x)= \left(\sum_{k\geq  0 } |\varphi_k({L})e^{itL}f(x)|^2\right)^{1/2}
\end{align}
for every $f\in L^2(X)$. Then from \eqref{xx}, it is direct to see that $\|e^{itL}f\|_p  \leq   C\| T_\varphi f \|_p$
for $2<p<p_0'$.

Next, we define a sharp maximal function ${\mathfrak  M}^{\#}_{T_\varphi, L, K}$ of $T_\varphi$ as follows: for every $K\in{\mathbb N}$ and every $f\in L^2(X)$,
\begin{eqnarray}\label{e3.6}
{\mathfrak  M}^{\#}_{T_\varphi, L, K}f(x)&=& \sup_{B\ni x} \left( \fint_{B} \big|T_\varphi(I-e^{-r_B^mL})^K f(y)\big|^2 d\mu(y)\right)^{1/2} ,
\end{eqnarray}
where the supremum is taken over all balls $B$ in $X$ containing $x$, and $r_B$ is the radius of $B.$
In order to prove Theorem~\ref{th1.1}, it suffices to show the following two arguments:

\begin{itemize}

\item[($\mathfrak{a}_1$)] the operator $T_\varphi$ satisfies condition \eqref{e3.1}  for every  $2<p<q<p'_0$
and $A_{r_B}=I-(I-e^{-r_B^mL})^K$ for every  $K\in{\mathbb N}$;

\item[($\mathfrak{a}_2$)] by choosing $K$ large enough,  for $s=n|1/2-1/p|,$
  \begin{eqnarray}\label{e3.7}
 \left\| {\mathfrak  M}^{\#}_{T_\varphi, L, K}f\right\|_{p} \leq C(1+|t|)^{s}
 \left(\sum\limits_{k\geq  0}  2^{ksp} \|\varphi_k({L})  f\|_p^p\right)^{1/p}.
\end{eqnarray}

\end{itemize}

Before we prove the above two arguments ($\mathfrak{a}_1$) and ($\mathfrak{a}_2$), let us show that Theorem~\ref{th1.1} is a straightforward consequence of
them. Indeed, when ($\mathfrak{a}_1$) holds for $T_\varphi$, it follows
 from  (b) of Proposition~\ref{prop2.5} and Proposition~\ref{prop3.2}  that for $2<p<p_0'$,
 $\|  {\mathfrak  M}_2\big(  T_\varphi f  \big)  \|_p
\leq  C_p    \big( \|f\|_p + \|{\mathfrak  M}^{\#}_{T_\varphi, L, K}f  \|_p \big)$. This, together with \eqref{e3.7}, yields that
\begin{eqnarray}\label{mmm}
\|e^{itL}f\|_p  \leq   C\| T_\varphi f \|_p
 &\leq&   C \|  {\mathfrak  M}_2\big(  T_\varphi f  \big)  \|_p\leq  C_p    \big( \|f\|_p + \|{\mathfrak  M}^{\#}_{T_\varphi, L, K}f  \|_p \big)\nonumber\\
&\leq& C\|f\|_p + C(1+|t|)^{s} \left(\sum\limits_{k\geq  0} 2^{ks p} \|\varphi_k({L})f\|_p^p\right)^{1/p}\nonumber\\
 &\leq&  C\|f\|_p + C(1+|t|)^{s}\left( \|\varphi_0(f)\|_p
 +  \left\|\left(\sum_{k> 0}   2^{2ks  }\left|\varphi_k({L})f\right|^2\right)^{1/2}\right\|_p\right)\nonumber\\
 &\leq&  C(1+|t|)^{s}  \left(\|f\|_p + \left\|\left(\sum_{k> 0}  \left| \phi_k({L}) \big[L^s f\big]\right|^2\right)^{1/2}\right\|_p\right)\nonumber\\
 &\leq&  C(1+|t|)^{s}  \left(\|f\|_p   +\|L^s f \|_p\right),
\end{eqnarray}
 where
 in the fifth inequality we have used the embedding $\ell^2\hookrightarrow \ell^p$ for $p\geq 2$,
  in the sixth inequality the function $\phi_k(\lambda)=\varphi(2^{-k}\lambda) (2^{-k}\lambda)^{-s}$, and in the last inequality
 we used (b) of Proposition~\ref{prop2.5} for the Littlewood-Paley result for functions in $L^p(X).$
This proves Theorem~\ref{th1.1}.

  \medskip


We now first prove the argument ($\mathfrak{a}_1$).
Indeed, in virtue  of
the formula
\begin{eqnarray}\label{e3.4}
 I-(I-e^{-r_B^mL})^K =\sum_{\tau=1}^K
 \left(
 \begin{array}{lcr}
K\\
 \tau
 \end{array}
 \right)(-1)^{\tau+1} e^{-\tau r_B^mL}
\end{eqnarray}
and  the commutativity property
 $\varphi_k ({L})e^{itL} e^{-\tau r_B^mL} =e^{-\tau r_B^mL} \varphi_k ({L})e^{itL}$,  it is enough to show that
for all ball $B$ containing $x,$
 \begin{eqnarray} \label{e3.5} \hspace{1cm}
\left(  \fint_B \left(\sum_{k\geq 0}
\big| e^{-\tau r_B^mL} \varphi_k({L}) e^{itL}f(y)\big|^2\right)^{q/2} d\mu(y)\right)^{1/q}
\leq C  {\mathfrak  M}_2 \big(  T_\varphi f  \big)(x).
\end{eqnarray}

Let us prove \eqref{e3.5}. From hypothesis  \eqref{GGE},  it is seen that  condition ${\rm (GGE_{2,q, m})}$ holds
for $2<p<q<p'_0$, i.e,
there exist  constants $C, c>0$  such that for every $u>0$ and $x, y\in X$,
\begin{equation}\label{pp}
\big\|P_{B(x, u^{1/m})} e^{-uL} P_{B(y, u^{1/m})}\big\|_{2\to {q}}\leq
C V(x,u^{1/m})^{-({\frac{1}{ 2}}-{1\over q})} \exp\left(-c\left({d(x,y)^m \over    u}\right)^{1\over m-1}\right).
\end{equation}
By
  Minkowski's inequality, \eqref{pp} and (ii) of Lemma~\ref{le2.1},
  conditions \eqref{e2.2} and \eqref{e2.3} for every
  $\tau =1,2, \ldots, K$ and every ball $B$ containing $x,$
  the left hand side of    \eqref{e3.5} is less than
\begin{eqnarray*}
&&\hspace{-1cm}V(B)^{-1/q}  \sum_{j=0}^{\infty}   \left(\sum_{k\geq 0} \big(
\|P_{B}e^{-\tau r_B^mL} P_{A(x_B, r_B, j)}\varphi_k({L})e^{itL}f \|_{q} \big)^{2}\right)^{1/2}\nonumber \\
&\leq&V(B)^{-1/q} \sum_{j=0}^{\infty}   \|P_{B}e^{-\tau r_B^mL} P_{A(x_B, r_B, j)}\|_{2\to q}
\left( \sum_{k\geq 0}
 \| \varphi_k({L})e^{itL}f \|_{L^2(A(x_B, r_B, j))}^{2}\right)^{1/2}\nonumber \\
&\leq&  C \sum_{j=0}^{\infty} \left({V((j+1)B) \over  V(B)}\right)^{1/2} e^{-c_\tau  j^{m/(m-1)}} (1+j)^n
\left(   \fint_{(j+1)B}
\sum_{k\geq 0} \big| \varphi_k({L})e^{itL}f(y)\big|^2 d\mu(y) \right)^{1/2}\nonumber \\
&\leq& C  \sum_{j=0}^{\infty} e^{-c_\tau j^{m/(m-1)}} (1+j)^{3n/2} {\mathfrak  M}_2 \big(  T_\varphi f  \big) (x) \nonumber \\
 &\leq& C  {\mathfrak  M}_2 \big(   T_\varphi f  \big) (x). \nonumber
\end{eqnarray*}
The above estimate  yields   \eqref{e3.5}.

Thus, we obtain that the argument ($\mathfrak{a}_1$) holds.

We now show the argument ($\mathfrak{a}_2$). In the sequel  we let ${\phi}\in C_0^{\infty}({\mathbb R})$ supported in $(1/4, 4)$ and ${\phi}(x)=1$ if $x\in (1/2, 2)$,
 and set $\phi_k(x)=\phi(2^{-k}x)$ for $k\geq1$. Let $\phi_0\in C_0^{\infty}([-4, 4])$ and ${\phi}_0(x)=1$ if $x\in (-2,  2)$.
By spectral theory,  we have that
$\varphi_k({L})  f ={\phi}_k( {L})\varphi_k({L}) f $ for $k\geq0$ and for every $f\in L^2(X).$
Hence,   the proof of \eqref{e3.7}  reduces to show that
 \begin{eqnarray}\label{e3.8}
 \|I\|_{p}+ \|II\|_{p}+  \|III\|_{p}  \leq  C(1+|t|)^{s}
 \left(\sum\limits_{k\geq  0} \|\varphi_k({L})  f \|_p^p\right)^{1/p},
\end{eqnarray}
where
 \begin{eqnarray*}
I(x)&=& \sup_{B\ni x}  \left(\fint_B  \sum_{0\leq k\leq -j}   2^{-2ks}   \left|  (I-e^{-r_B^mL})^K \phi_k({L})
 \big[e^{itL} \phi_k({L}) \varphi_k({L})  f\big]   (y) \right|^2 d\mu(y)\right)^{1/2},\nonumber\\
II(x) &=&   \sup_{B\ni x}   \left(  \fint_B  \sum_{\substack{ k+j>0\\  j\geq (m-1)k +m{\rm log}_2 (2+2|t|) \\ k\geq 0}}   2^{-2ks}   \left|  (I-e^{-r_B^mL})^K
 e^{itL} \phi_k({L}) \big[\varphi_k({L})  f\big]  (y) \right|^2 d\mu(y)\right)^{1/2}, \nonumber\\
III(x) &=&    \sup_{B\ni x}     \left(\fint_B  \sum_{\substack{ k+j>0\\ j< (m-1)k +m{\rm log}_2 (2+2|t|) \\  k\geq  0}}
  2^{-2ks}   \left|  (I-e^{-r_B^mL})^K
 e^{itL}\phi_k({L}) \big[\varphi_k({L})  f\big]   (y) \right|^2 d\mu(y)\right)^{1/2}.
\end{eqnarray*}
Here,  we use the notation in  the above decomposition that the ball $B$   is centered at $x_B$ and its radius $r_B$
 is in $[2^{(j-1)/m}, 2^{j/m})$ for some $j\in{\mathbb Z}$.

\medskip

\noindent
{\underline{Estimate  of the term $I(x)$}.} \
By the Minkowski inequality, we see that
\begin{align*}
I(x)&\leq \sup_{B\ni x } \bigg( \fint_B \Big|(I-e^{-r_B^mL})^K e^{itL} \phi_0(L)[\varphi_0(L)f](y)\Big|^2d\mu(y) \bigg)^{  1\over2}
\\
&\quad+\sup_{B\ni x }   \mu(B)^{-1/2}   \sum_{u=0}^{\infty} \sum_{1\leq  k\leq -j}  2^{- ks}
\left\|P_{B}  (I- e^{-r^m_BL})^K  {\phi}_k({L})  P_{A(x_B, r_{B }, u)}  [e^{itL}  {\phi}_k( {L})
 {\varphi}_k({L})f] \right\|_2\\
 &=I_1(x)+I_2(x).
\end{align*}

For the term $I_1(x)$, from the arguments in \eqref{e3.4} and \eqref{e3.5},  it is direct to see that for every $x\in B$,
$ I_1(x)\leq C{\mathfrak  M}_2\big(  e^{itL} \phi_0(L)[\varphi_0(L)f]\big)(x). $
Then from Proposition \ref{prop1.1},
$$ \|I_1\|_p\leq  C\big\|  e^{itL} \phi_0(L)[\varphi_0(L)f]\big\|_p \leq C(1+|t|)^s\|\varphi_0(L)(f)\|_p.$$

For the term $I_2(x)$,
since the function
$ (1- e^{-r^m_B\lambda})^K \phi_k(\lambda) $ is supported in $  [2^{k-2}, \ 2^{k+2}]$, $k\geq1$,
it tells us that for $u=0, 1, $
\begin{eqnarray*}
 \big\|P_{B }  (I- e^{-r_B^m L})^K   \phi_k ({L}) P_{A(x_B, r_{B }, u)} \big\|_{2\to 2}
 &\leq&   \big\| (I- e^{-r_B^m L})^K   \phi_k ({L})   \big\|_{2\to 2} \nonumber\\
 &\leq&  C    \|   (1- e^{-r_B^m\lambda})^K
 \phi_k(\lambda)  \|_{L^{\infty}} \nonumber\\
 &\leq&   C\min\{ 1, (2^kr_B^m)^{K} \},
\end{eqnarray*}
also for $u\geq 2$,  we  use Lemma~\ref{le2.2} to obtain that for every $M>0,$
\begin{eqnarray} \label{e3.9}
 \big\|P_{B }  (I- e^{-r_B^m L})^K    \phi_k ({L}) P_{A(x_B, r_{B }, u)}\big\|_{2\to 2}
 &\leq&    C u^{-M}\big(2^{k/m} r_{B }\big)^{-M-n} \|\delta_{2^{k+2}}\big( (1- e^{-r_B^m\lambda})^K
 \phi_k(\lambda) \big) \|_{  W^{M+n+1}_2}\nonumber\\
 &\leq&   Cu^{-M} 2^{-(k+j)(M+n)/m} \|  (1- e^{-2^{ (k+2)}r_B^m\lambda})^K  \phi(4\lambda)  \|_{  W^{M+n+1}_2}
 \nonumber\\
 &\leq&    C u^{-M} \min\{2^{-(k+j)(M+n)/m}, 2^{(k+j)(K-M/m-n/m)} \}.
\end{eqnarray}
Those, in combination with $k+j\leq 0$ and the fact that for all $u\geq 0$ and $g\in L_{loc}^2(X)$
\begin{eqnarray} \label{e2.555}
 \|P_{A(x_B, r_{B }, u)}g  \|_2 &\leq&  \mu((u+1)B)^{1/2} \left(\fint_{(u+1)B }|  g(y)|^2 d\mu(y)\right)^{1/2}
 \nonumber\\
 &\leq& C (1+u)^{n/2} \mu(B)^{1/2}{\mathfrak  M}_2\left(  g \right)(x),
\end{eqnarray}
yield
\begin{eqnarray*}
I_2(x)
   &\leq&      \sup_{B\ni x}    \sum_{1\leq  k\leq -j}  \sum_{u=0}^{\infty} 2^{- ks}   (1+u)^{-(M-n/2)} 2^{(k+j)(K-(M+n)/m)}
   {\mathfrak  M}_2\Big(  e^{itL}  {\phi}_k( {L})
 {\varphi}_k({L})f \Big)  (x)\\
     &\leq&
    C
    \sup_{B\ni x}
	 \sum_{1\leq  k\leq -j}  2^{- ks}    2^{(k+j)(K-(M+n)/m)}
	 {\mathfrak  M}_2\Big(  e^{itL}  {\phi}_k( {L})
 {\varphi}_k({L})f \Big)  (x),
\end{eqnarray*}
  where  $M>n/2$ and $K$ is large enough so that $K> (M+n)/m.$
We then use the embedding $\ell^p\hookrightarrow \ell^{\infty}$, the Minkowski inequality, $L^{p/2}$-boundedness of ${\mathfrak  M}$
and Proposition~\ref{prop1.1} to see that
 \begin{eqnarray*}
  \|I_2\|_{p}
   &\leq&   C \left\| \left(\sum_{j=-\infty}^{\infty} \left( \sum_{1\leq k\leq -j}  2^{(k+j)(K-(M+n)/m)}  2^{- ks}
 {\mathfrak  M}_2\left(  e^{itL}  {\phi}_k({L})
 {\varphi}_k({L})f \right)  (x)\right)^p\right)^{1/p}\right\|_p\\
   &\leq&  C \sum_{\ell\geq 0} 2^{-\ell(K-(M+n)/m)}\left(\sum_{j< -\ell}     2^{ (\ell +j)sp}
   \left\| {\mathfrak  M}_2\left( e^{itL}  {\phi}_{-(\ell+j)}({L})
 {\varphi}_{-(\ell+j)}({L})f  \right) \right\|_p^p
    \right)^{1/p} \\
	&\leq&  C \sum_{\ell\geq 0} 2^{-\ell(K-(M+n)/m)}\left(\sum_{j< -\ell}     2^{(\ell +j)sp}
   \left\|  e^{itL}\phi_{-(\ell+j)}({L})\big[\varphi_{-(\ell+j)}({L})f\big]  \right\|_p^p
    \right)^{1/p} \\
   &\leq& C(1+|t|)^s\sum_{\ell\geq 0} 2^{-\ell(K-(M+n)/m)}\left(  \sum_{j<-\ell}
  \left\|   \varphi_{-(\ell+j)} ({L}) f    \right\|_{p}^p\right)^{1/p} \nonumber\\
 &\leq&   C(1+|t|)^s\left(  \sum_{k\geq  1}
  \left\|   \varphi_k({L})f \right\|_{p}^p\right)^{1/p}
\end{eqnarray*}
 as desired,  as long as $K$ is chosen large enough so that $K> (M+n)/m$. Combining the estimates of $I_1$ and $I_2$ we get that
 $$\|I\|_{p}\leq  C(1+|t|)^{s}
 \left(\sum\limits_{k\geq  0} \|\varphi_k({L})  f \|_p^p\right)^{1/p}.
$$

  \medskip

\noindent
{\underline{Estimate of  the term $II(x)$.}} \ \  Note that
\begin{eqnarray*}
 II(x)&\leq & \sup_{B\ni x } \bigg( \fint_B \Big|(I-e^{-r_B^mL})^K e^{itL} \phi_0(L)[\varphi_0(L)f](y)\Big|^2d\mu(y) \bigg)^{  1\over2}
\\
&&+\sup_{B\ni x}      \sum_{\substack{ k+j>0\\  j\geq (m-1)k +m{\rm log}_2 (2+2|t|) \\ k\geq  1}}
 \sum_{\ell=0}^{\infty} 2^{- ks}  \mu(B)^{-1/2}  \nonumber\\
&&\times \left\|P_B (I-e^{-r_B^mL})^K  e^{itL} \phi_k({L})P_{A(x_B, r_B, \ell)}\right\|_{2\to 2}
  \left\|P_{A(x_B, r_B, \ell)}
   [\varphi_k({L})f   ] \right\|_2\\
  &=& II_1(x)+II_2(x).
\end{eqnarray*}

Similar to the estimate of $I_1(x)$ above, we see that
$ \|II_1\|_p \leq C(1+|t|)^s\|\varphi_0(L)(f)\|_p.$

We now estimate $II_2(x)$. For a fixed $r_B>0$, we   choose
 a sequence of points  $\{x_i\}_{i}  \subset X$ such that
$d(x_{i}, x_k)> r_B$ for $ {i}\neq  k$ and $\sup_{x\in X}\inf_{i} d(x, x_{i})
\le r_B $. Such sequence exists because $X$ is separable. Set
$$
J_{\ell} =\big\{B(x_{i}, r_B):  B(x_{i}, r_B)\cap A(x_B, r_B, \ell) \not=\emptyset \big\},\quad \ell\geq0.
$$
 It follows from \eqref{e200} that for every $B(x_{i}, r_B)\in J_{\ell},$
 $$
 V(x_B, r_B)\leq \left(1+ {d(x_i, x_B)\over r_B}\right)^D V(x_i, r_B)\leq C(1+\ell)^D V(x_i, r_B)
 $$
 and  so
\begin{eqnarray} \label{qq}
 \# J_{\ell}
 \le
  C (1+\ell)^D \times  {V(x_B,  (\ell+1)r_B )  \over   V(x_B, r_B)}\leq C (1+\ell)^{D+n}< \infty.
\end{eqnarray}
Then we have
\begin{eqnarray*}
 II_2(x)&\leq &\sup_{B\ni x}      \sum_{\substack{ k+j>0\\  j\geq (m-1)k +m{\rm log}_2 (2+2|t|) \\ k\geq  1}}
 \sum_{\ell=0}^{\infty} \sum_{B(x_{i}, \, r_B)\in J_\ell}2^{- ks}  \mu(B)^{-1/2}  \nonumber\\
&&\times \left\|P_B (I-e^{-r_B^mL})^K  e^{itL} \phi_k({L})P_{B(x_i,\, r_B)}\right\|_{2\to 2}
  \left\|P_{A(x_B, r_B, \ell)}
   [\varphi_k({L})f   ] \right\|_2.
\end{eqnarray*}
In this case, since $j\geq (m-1)k +m{\rm log}_2 (2+2|t|)$ and so $r_B\geq c2^{(m-1)k/m} (1+|t|)$ with $c=2^{(m-1)/m}\geq 1/4$, we apply
   Proposition~\ref{prop3.3} to see that for every $B(x_{i}, \, r_B)\in J_\ell $ with $\ell\geq 7, 8, \cdots, $
  \begin{eqnarray} \label{kk}\hspace{1cm}
   \left\|P_B (I-e^{-r_B^mL})^K  e^{itL} \phi_k({L})P_{B(x_i,\, r_B)}\right\|_{2\to 2}\leq C
   \left(1+{d(B, \, B(x_i,\, r_B))\over 2^{(m-1)k/m} (1+|t|)}\right)^{-M}\leq C\big( 1+\ell\big)^{-M}
\end{eqnarray}
   for every $M>0$. For $\ell=0, 1, \cdots, 6$, it follows from  $L^2$-boundedness of $(I-e^{-r_B^mL})^K  e^{itL} \phi_k({L})$ that
$
   \left\|P_B (I-e^{-r_B^mL})^K  e^{itL} \phi_k({L})P_{B(x_i,\, r_B)}\right\|_{2\to 2}\leq C.
$
   These, in combination with the fact that for every $x\in B$,
\begin{eqnarray*}
 \|P_{A(x_B, r_{B }, \ell )}[\varphi_k({L})f   ]  \|_2 &\leq&  \mu((\ell +1)B)^{1/2} \left(\fint_{(\ell +1)B }|  \varphi_k({L})f  (y)|^2 d\mu(y)\right)^{1/2}
 \nonumber\\
 &\leq& C (\ell +1)^{n/2} \mu(B)^{1/2}{\mathfrak  M}_2\left(  \varphi_k({L})f   \right)(x),
\end{eqnarray*}
imply
 \begin{eqnarray*}
II_2(x) &\leq&      C  \sum_{  k\geq 1}
 \sum_{\ell=0}^{\infty}  2^{- ks}    \big( 1+\ell\big)^{-(M-D-3n/2)}
  {\mathfrak  M}_2\left(\varphi_k({L}) f \right)(x) \\
  &\leq& C\sum\limits_{k\geq  1}   2^{-ks}
  {\mathfrak  M}_2\left(\varphi_k({L}) f \right)(x)
\end{eqnarray*}
  as long as $M$ in \eqref{kk} is chosen large enough so that $M> D+2n$. As a consequence, we have that for $2<p<p'_0,$
 \begin{eqnarray*}
 \|II_2 \|_{p}  \leq   C\left\| \sum_{k\geq 1}  2^{- ks}
	 {\mathfrak  M}_2\left(\varphi_k({L}) f \right)\right\|_{p}
 & \leq &      C \left (\sum_{k\geq  1}
 \left\|  {\mathfrak  M}_2\Big(\varphi_k({L}) f\Big) \right\|_{p}^{p} \right)^{1/{p}}
  \leq     C \left (\sum_{k\geq  1}
 \left\|   \varphi_k({L}) f  \right\|_{p}^{p} \right)^{1/{p}}.
\end{eqnarray*}

Combining the estimates of $II_1$ and $II_2$ we obtain the estimate of $II$
 as desired.

\medskip

\noindent
{\underline{Estimate of  the term $III(x)$.}} \ As to be seen later, the term $III(x)$
is the major one.

Similar to the estimates for $II$ and $I$ above, we write
\begin{align*}
III(x) &\leq \sup_{B\ni x } \bigg( \fint_B \Big|(I-e^{-r_B^mL})^K e^{itL} \phi_0(L)[\varphi_0(L)f](y)\Big|^2d\mu(y) \bigg)^{  1\over2}
\\
&\quad+  \sup_{B\ni x}     \left(\fint_B  \sum_{\substack{ k+j>0\\ j< (m-1)k +m{\rm log}_2 (2+2|t|) \\  k\geq  1}}
  2^{-2ks}   \left|  (I-e^{-r_B^mL})^K
 e^{itL}\phi_k({L}) \big[\varphi_k({L})  f\big]   (y) \right|^2 d\mu(y)\right)^{1/2}\\
&=III_1(x)+III_2(x) .
\end{align*}
Again, it is clear that $ \|III_1\|_p \leq C(1+|t|)^s\|\varphi_0(L)(f)\|_p.$ It suffices to verify $III_2(x)$.

For a given $x\in X$ and a ball $x\in B_j=B(x_{B_j}, r_{B_j}) $ with $r^m_{B_j}\in [2^{j-1}, 2^{j}]$.
We define a family of operators $\{A_{r_{B_j}}\}_{j=1}^{\infty}$ with  non-negative kernels
$\{a_{r_{B_j}}(x,y)\}_{j=1}^{\infty} $  such that
\begin{eqnarray*}
a_{r_{B_j}}(x,y)= {1\over \mu(B(x, 2r_{B_j}))} \chi_{B(x, 2r_{B_j})}(y).
 \end{eqnarray*}
We will use
$$
A_{r_{B_j}}g(x)=\int_X  a_{r_{B_j}}(x,y)  g(y)  d\mu(y)
$$
to replace the mean value $\fint_{B_j}$ in the term $III_2(x)$. It is seen  that for every non-negative function $g\in L^1_{\rm loc}(X)$
and $B_j$ containing $x,$
\begin{eqnarray*}
 \fint_{B_j}  g(y) d\mu(y)
&\leq&    \left({  \mu(B(x_{B_j}, 3r_{B_j}))\over \mu(B_j) }\right)
A_{r_{B_j}}g(x)
 \leq   C
A_{r_{B_j}}g(x)
 \end{eqnarray*}
and so   $ III_2(x) \leq C {\widetilde {III}}_2(x),$
 where
 \begin{eqnarray}\label{e3.15} \hspace{1cm}
 {\widetilde {III}}_2(x):&=& \sup_{j\in{\mathbb Z}}
  \left(\sum_{ \substack{  k+j>0\\  j< (m-1)k +m{\rm log}_2 (2+2|t|) \\ k\geq  1}}     2^{-2ks} A_{r_{B_j}}\left(
 \left | (I-e^{-r_{B_j}^mL})^Ke^{itL} \phi_k({L})  [\varphi_k({L})f]\right|^2 \right)(x)\right)^{1/2}.
\end{eqnarray}

Now for  every  $k\geq 1$, we   choose
 a sequence $(x^{(k)}_{\tau})_\tau  \in X$ such that
$d(x^{(k)}_{\tau}, x^{(k)}_\ell)> 2^{k(m-1)/m}(1+|t|)$ for $ {\tau}\neq  \ell$ and $\sup_{x\in X}\inf_{\tau} d(x,x^{(k)}_{\tau})
\le 2^{k(m-1)/m}(1+|t|)$. Such sequence exists because $X$ is separable.
Let $B_{\tau}^{(k), \ast}=B(x^{(k)}_{\tau}, 8 \cdot 2^{k(m-1)/m}(1+|t|))$ and define $ {B^{(k)}_{\tau}}$ by the formula
$$
{B^{(k)}_{\tau}}=\bar{B}\left(x^{(k)}_{\tau},2^{k(m-1)/m}(1+|t|)\right)\setminus
\bigcup_{\ell<{\tau}}\bar{B}\left(x^{(k)}_\ell, 2^{k(m-1)/m}(1+|t|)\right),$$
where $\bar{B}\left(x^{(k)}_{\tau}, r \right)=\{y\in X \colon d(x^{(k)}_{\tau},y)
\le r\}$.
 We cover $X$ by a grid ${\mathscr R}_k$ consisting of such $\{{B^{(k)}_{\tau}}\}_{\tau}$, that is,
 $X=\bigcup_{ B_{\tau}^{(k)} \in {\mathscr R}_{k}  } B_{\tau}^{(k)}$.
 For every ${B^{(k)}_{\tau}}\in {\mathscr R}_k$, we denote  by $f^{B^{(k)}_{\tau}}=f\chi_{B^{(k)}_{\tau}}.$ Hence,
 one writes
 \begin{eqnarray}\label{e3.16}
 {\widetilde {III}}_2(x)
\leq III_{21}(x)+III_{22}(x),
\end{eqnarray}
where
$$
III_{21}(x)= \sup_{j\in\mathbb Z}  \left(  \sum_{ \substack{  k+j>0\\  j< (m-1)k +m{\rm log}_2 (2+2|t|) \\ k\geq 1}}   2^{-2ks}
 A_{r_{B_j}}\left(\left|\sum_{{B^{(k)}_{\tau}}\in {\mathscr R}_k} \bchi_{B_{\tau}^{(k), \ast}}(I-e^{-r_{B_j}^mL})^Ke^{itL} \phi_k({L})
 [\varphi_k({L})f ]^{B^{(k)}_{\tau}}\right|^2
 \right)(x)\right)^{1/2}  $$
 and
  $
III_{22}(x)$
is the analogous expression where $ \bchi_{B^{(k), \ast}_{\tau}}$ is replaced with  $\bchi_{X\backslash B^{(k), \ast}_{\tau}}.$

Let us first estimate the  term $III_{21}(x)$. Using the embedding $\ell^p\to \ell^{\infty}$, the bounded overlap of $B_{\tau}^{(k), \ast}$ and   Minkowski's
inequality, we obtain that the $L^p$-norm  of the term $III_{21}(x)$ is less than
  \begin{eqnarray*}
C \left( \sum_{j\in\mathbb Z}
 \left\|   \left(\sum_{ \substack{  k+j>0\\  j< (m-1)k +m{\rm log}_2 (2+2|t|)\\ k\geq 1}} 2^{-2ks} A_{r_{B_j}}\left(
  \sum_{{B^{(k)}_{\tau}}\in {\mathscr R}_k} \bchi_{ B_{\tau}^{(k), \ast}} \left  | (I-e^{-r_{B_j}^mL})^Ke^{itL}
 \phi_k({L}) [ \varphi_{k}({L})f]^{B^{(k)}_{\tau}}\right|^2 \right)  \right)^{1/2} \, \right\|^p_{p} \right)^{1/p}.
\end{eqnarray*}

To continue, we claim that the supports of the functions $\{A_{r_{B_j}}\big(\bchi_{B^{(k), \ast}_{\tau}})\}_{\tau}$
have bounded overlap, uniformly in $k.$ Assume this  at the moment. Then by
setting $\ell=k+j>0$,  applying Minkowski's inequality, and the above claim, we obtain
that  $$\|III_{21}\|_{p}\leq \sum_{\ell>0} E_{\ell},
$$
where
\begin{eqnarray*}
E_{\ell}:= \left( \sum_{j<\ell } \sum_{
B^{(\ell-j)}_{\tau}\in {\mathscr R}_{\ell-j}  }
 2^{-(\ell-j)sp} \left\|A_{r_{B_j}}  \bchi_{B_{\tau}^{(\ell-j), \ast}}\left(\left |(I-e^{-r_{B_j}^mL})^Ke^{itL}
 \phi_{\ell-j}({L}) \left[\varphi_{\ell-j}({L})f\right]^{B^{(\ell-j)}_{\tau}} \right |^2
    \right) \right\|^{p/2}_{p/2} \right)^{1/p}.
\end{eqnarray*}

We now show the claim. Note that  for ${B^{(k)}_{\tau}}\in {\mathscr R}_k$, ${B^{(k), \ast}_{\tau}}$ has radius   $8\cdot 2^{k(m-1)/m}(1+|t|)$.
It follows from $r_{B_j}\leq 2^{j/m}\leq 2^{k(m-1)/m-2}(1+|t|)$ that for fixed $k$,
$
A_{r_{B_j}}\big(\bchi_{B^{(k), \ast}_{\tau}})(x) \cdot A_{r_{B_j}}\big(\bchi_{B^{(k), \ast}_{\ell}})(x)=0
$
when $d(x^{(k)}_{\tau}, x^{(k)}_\ell)\geq 20\cdot 2^{k(m-1)/m}(1+|t|).$
From  \eqref{e200}, we  know that
\begin{eqnarray*}
 V(x^{(k)}_\ell, 2^{k(m-1)/m}(1+|t|))&\leq& \left(1+ {d(x^{(k)}_\ell, x^{(k)}_{\tau})\over r_B}\right)^D  V(x^{(k)}_{\tau}, 2^{k(m-1)/m}(1+|t|))\\
 &\leq& C  V(x^{(k)}_{\tau}, 2^{k(m-1)/m}(1+|t|)),
\end{eqnarray*}
which implies
$$
\sup_\tau\#\{\ell:\;d(x^{(k)}_{\tau},x^{(k)}_\ell)\le 30 \cdot 2^{{k(m-1)\over m}}(1+|t|)\} \le
  \sup_x  {V(x,  30 \cdot 2^{{k(m-1)\over m}} (1+|t|)  )\over
  V(x,2^{{k(m-1)\over m}-2}(1+|t|))}\leq C  < \infty.
$$


Next we will show  that
\begin{eqnarray}\label{e3.17}
E_{\ell}\leq C(1+|t|)^s2^{-\ell s/m}
 \left(\sum_{k>  0} \|\varphi_k({L})  f\|_p^p\right)^{1/p}.
\end{eqnarray}
Once \eqref{e3.17} is proven, we see that
\begin{eqnarray}\label{e3.18}
\|III_{21}\|_{p} \leq  C(1+|t|)^s \left(\sum\limits_{k\geq 1}
 \|\varphi_{k}({L})f \|_p^p\right)^{1/p}.
\end{eqnarray}

Let us prove estimate \eqref{e3.17}. First, we observe that for every $g\in L^1(X)$ and ${p/2}>1,$
\begin{eqnarray}\label{e3rr}
\|A_{r_{B_j}}\left(\bchi_{ B_{\tau}^{(\ell-j), \ast}} g\right)\|_{p/2}
&\leq&
 \left(\sup_{y\in B_{\tau}^{(\ell-j), \ast}} \int_{X}
 a^{p/2}_{r_{B_j}}(x,y) \bchi_{ B_{\tau}^{(\ell-j), \ast}}(y) d\mu(x)\right)^{2/p} \|g \|_{1}\nonumber\\
 &\leq& C \sup_{y\in B_{\tau}^{(\ell-j), \ast}} [V(y, r_{B_j})^{-(1-{2\over p})}]
  \|g \|_{1}.
 \end{eqnarray}
From this, we see that the term $E_{\ell}$ is dominated by a constant multiple of
 $$
\left( \sum_{j<\ell }\sum_{
B^{(\ell-j)}_{\tau}\in {\mathscr R}_{\ell-j}   }
 2^{-(\ell-j)sp}  \sup_{y\in B_{\tau}^{(\ell-j), \ast}} [V(y, r_{B_j})^{-(p/2-1)}]  \left\|(I-e^{-r_{B_j}^mL})^Ke^{itL}
 \phi_{\ell  -j}({L}) [ \varphi_{\ell-j}({L})f]^{B^{(\ell-j)}_{\tau}} \right\|_{2}^p   \right)^{1/p}.
 $$
Since  the operator $(I-e^{-r_{B_j}^mL})^Ke^{itL}
 \phi_{\ell-j}({L}) $ is uniformly bounded on $L^2(X)$ and $[\varphi_{\ell-j}({L})f]^{B^{(\ell-j)}_{\tau}}$  is supported on the ball
 $B^{(\ell-j)}_{\tau}$, we see by the H\"older inequality that
the term $E_{\ell}$ is controlled by a constant multiple of
\begin{eqnarray*}
&&\hspace{-1cm}\left( \sum_{j<\ell }    2^{-(\ell-j) s p}
 \sum_{
B^{(\ell-j)}_{\tau}\in {\mathscr R}_{\ell-j}  }
\sup_{y\in B_{\tau}^{(\ell-j), \ast}} \left( \mu(B^{(\ell-j)}_{\tau}) \over  \mu(B(y, r_{B_j}))\right)^{{p\over 2}-{1 }}
\left\| \left[ \varphi_{\ell-j}({L})f\right]^{B^{(\ell-j)}_{\tau}}\right\|^p_p\right)^{1/p}.
\end{eqnarray*}
Note that   for $y\in B_{\tau}^{(\ell-j), \ast}$,
$$
 \left( \mu(B^{(\ell-j)}_{\tau}) \over  \mu(B(y, r_{B_j}))\right)  \leq C (1+|t|)^n 2^{{n\over m}[(\ell-j)(m-1) -j]},
$$
 which yields
\begin{eqnarray*}
E_{\ell}&\leq& C (1+|t|)^{n({1\over 2}-{1\over p})}
\left(\sum_{j<\ell }    2^{-n(\ell-j)({p\over 2}-{1 })} 2^{{n\over m}[(\ell-j)(m-1) -j]({p\over 2}-{1 }) }
    \sum_{
B^{(\ell-j)}_{\tau}\in {\mathscr R}_{\ell-j}  }
\left\| \left[ \varphi_{\ell-j}({L})f\right]^{B^{(\ell-j)}_{\tau}}\right\|^p_p \right)^{1/p}\nonumber\\
 &=& C(1+|t|)^s 2^{-\ell s/m} \left( \sum_{j <\ell}
 \sum_{B^{(\ell-j)}_{\tau}\in {\mathscr R}_{\ell-j}} \left\| \left[ \varphi_{\ell-j}({L})f\right]^{B^{(\ell-j)}_{\tau}} \right\|^p_p\right)^{1/p}.
\end{eqnarray*}
 After summation in $B^{(\ell-j)}_{\tau}\in {\mathscr R}_{\ell-j}$, we  obtain
\begin{eqnarray*}
E_{\ell}
  &\leq&
 C(1+|t|)^s 2^{-\ell s/m} \left( \sum_{j <\ell}
  \left\|   \varphi_{\ell-j}({L})f \right\|^p_p\right)^{1/p}
  \leq
 C(1+|t|)^s 2^{-\ell s/m}
\left(\sum_{k\geq 1} \| \varphi_k({L})  f\|_p^p\right)^{1/p}.
\end{eqnarray*}
 This finishes the proof of \eqref{e3.17} and concludes the desired estimate \eqref{e3.18}  for the term  $III_{21}$.

  Concerning  the term $III_{22}$, we   use the embedding $\ell^p\to \ell^{\infty}$ and the Minkowski inequality to see that
the term $\|III_{22}\|_p$ is controlled by
\begin{eqnarray*}
\left(\sum_{j\in {\mathbb Z}}\left[
   \sum_{ \substack{  k+j>0\\ j< (m-1)k +m{\rm log}_2 (2+2|t|) \\ k\geq  1}}   2^{-2ks}
\left\| A_{r_{B_j}}\left(\left|\sum_{{B^{(k)}_{\tau}}\in {\mathscr R}_k} \bchi_{X\backslash B^{(\ell-j), \ast}_{\tau}}
 (I-e^{-r_{B_j}^mL})^Ke^{itL} \phi_k({L})
 [\varphi_k({L})f ]^{B^{(k)}_{\tau}}\right|^2
 \right) \right\|_{{p/2}} \right]^{p/2}\right)^{1/p}.
 \end{eqnarray*}
The proof of Theorem~\ref{th1.1} will be done if we can show that
  \begin{eqnarray} \label{ecc} \hspace{1cm}
&&\hspace{-3.5cm}\left\|A_{r_{B_j}}  \left(\left | \sum_{
B^{(k)}_{\tau}\in {\mathscr R}_{k}  }\bchi_{X\backslash B^{(k), \ast}_{\tau}}(I-e^{-r_{B_j}^mL})^Ke^{itL}
 \phi_{k}({L}) \left[\varphi_{k}({L})f\right]^{B^{(k)}_{\tau}} \right |^2
    \right)\right\|_{p/2} \nonumber\\
	& \leq & C (1+|t|)^{n(1-{2\over p})} 2^{{n\over m}[k(m-1)-j](1-{2\over p})} \left\|  \varphi_{k}({L})f
 \right\|^{2}_{p}
\end{eqnarray}
since from it, we recall that $s=n|{1/2}-{1/p}|$ to see  that
  \begin{eqnarray} \label{evv}
 \|III_{22} \|_{p}
 &\leq& C   (1+|t|)^{s }\left(\sum_{j\in {\mathbb Z}}\left[
   \sum_{ \substack{  k+j>0\\  j< (m-1)k + m {\rm log}_2 (2+2|t|)\\ k\geq 1}}   2^{-2ks} 2^{{n\over m}[k(m-1)-j](1-{2\over p}) }
\left\|  \varphi_{k}( {L})f
 \right\|^{2}_{p}  \right]^{p/2}\right)^{1/p}\nonumber\\
   &\leq& C (1+|t|)^{s } \sum_{\ell>0}
 \left( \sum_{j<\ell }
  2^{-n(\ell-j)({p\over 2}-1)}   2^{{n\over m}[(\ell-j)(m-1)-j]({p\over 2}-1) }  \left\|    \varphi_{\ell-j}( {L})f
     \right\|_{p}^{p}\right)^{1/p}\nonumber\\	
	   &=& C (1+|t|)^{s } \sum_{\ell>0} 2^{-{\ell s/m}  }
 \left( \sum_{j <\ell}  \left\|    \varphi_{\ell-j}( {L})f      \right\|_{p}^{p}\right)^{1/p}\nonumber\\	
    &\leq& C (1+|t|)^{s }  \left(\sum_{k\geq 1} \|\varphi_{k}( {L})f\|_p^p\right)^{1/p}.
\end{eqnarray}

It remains to prove \eqref{ecc}.
Observe  that $j< (m-1)k +m{\rm log}_2 (2+2|t|)$, and   $r_{B_j}\leq 2^{{(m-1)k/ m} +1}(1+|t|).$ Fix $x\in X$, $k\geq1$
and $j\in {\mathbb Z}$, we consider the  following three cases of $x^{(k)}_{\tau}$:

\medskip

\noindent
{\it Case 1:  $ d( x^{(k)}_{\tau}, x) \leq 6\cdot 2^{(m-1)k/m}(1+|t|)$.}

\smallskip

In this case,   for any $z\in B(x, 2r_{B_j}),$
$$
d(z,   x^{(k)}_{\tau}) \leq  d(z,   x) + d(  x^{(k)}_{\tau}, x)
\leq 8\cdot 2^{(m-1)k/m}(1+|t|) ;
$$
and so $B(x, 2r_{B_j})\cap  \Big(X\backslash B^{(k), \ast}_{\tau}\Big)=\emptyset;$

\medskip

\noindent
{\it Case 2: $ d(  x^{(k)}_{\tau}, x) \geq 10\cdot 2^{(m-1)k/m}(1+|t|)$.}

\smallskip

In this case,   for any $z\in  B(x, 2r_{B_j}) $
$$d(z,  x^{(k)}_{\tau}) \geq  d(  x^{(k)}_{\tau}, x) - d(z,   x)\geq 8\cdot 2^{(m-1)k/m}(1+|t|),
$$
and so  $B(x, 2r_{B_j})\subseteq  X \backslash B^{(k), \ast}_{\tau}$;

\medskip

\noindent
{\it Case 3:  $ 6\cdot 2^{(m-1)k/m}(1+|t|) \leq  d(  x^{(k)}_{\tau}, x ) \leq 10 \cdot 2^{(m-1)k/m}(1+|t|)$.}

\smallskip

In this case, we see that $ d({B^{(k)}_{\tau}}, B(x, 2r_{B_j})) \geq   2^{(m-1)k/m}(1+|t|)$, and
 \begin{eqnarray} \label{ll}\hspace{0.5cm}
 &&\hspace{-2cm}\sharp \left\{ \tau:  6\cdot 2^{(m-1)k/m}(1+|t|) \leq   d(x^{(k)}_{\tau}, x ) \leq 10 \cdot 2^{(m-1)k/m}(1+|t|) \right\} \nonumber\\
&\le&
  \sup_x  {V(x, 10\cdot 2^{(m-1)k/m+1}(1+|t|) )\over
  V(x,2^{ (m-1)k/m-2}(1+|t|))}
  \leq   C   < \infty.
 \end{eqnarray}

From {\it Cases 1, 2}  and   {\it 3}, we see that  there exists a  constant $C>0$ independent of $x $ and $j $  such that
 \begin{eqnarray*}
  A_{r_{B_j}}\left(
 \left|\sum_{{B^{(k)}_{\tau}}\in {\mathscr R}_{k}} \bchi_{X\backslash B^{(k), \ast}_{\tau}}(I-e^{-r_{B_j}^mL})^Ke^{itL}
 \phi_{k}({L})
  \left[ \varphi_{k}({L})f\right]^{B^{(k)}_{\tau}}\right|^2
 \right)(x) \leq D_1(x)+ C   D_2(x),
\end{eqnarray*}
where
 \begin{eqnarray*}
  D_1(x):&= &A_{r_{B_j}}\left(
 \left| (I-e^{-r_{B_j}^mL})^Ke^{itL} \phi_{k}({L})
  \left(\sum_{\tau: \,  d( x^{(k)}_{\tau}, \, x ) \geq 10\cdot 2^{(m-1)k/m}(1+|t|)}
  \left[ \varphi_{k}({L})f\right]^{B^{(k)}_{\tau}}\right)\right|^2
 \right)(x)
\end{eqnarray*}
and
 \begin{eqnarray*}
  D_2(x):= \left(\sum_{\tau: \,  6\cdot 2^{(m-1)k/m}(1+|t|) \leq  d( x^{k}_{\tau}, \, x) \leq 10 \cdot 2^{(m-1)k/m}(1+|t|) } A_{r_{B_j}}\left(
 \left|(I-e^{-r_{B_j}^mL})^Ke^{itL} \phi_{k}({L})
  \left[ \varphi_{k}({L})f\right]^{B^{(k)}_{\tau}}\right)\right|^2
 \right)(x).
\end{eqnarray*}

Let us estimate the term $D_1(x)$ by adapting  an argument as in the term $E_{\ell}$.
First note that
$$X=\bigcup_{ B_{\tau_1}^{(k)} \in {\mathscr R}_{k}  } B_{\tau_1}^{(k)}.$$
Then we write
 \begin{eqnarray*}
  D_1(x)&\leq  &  \sum_{B_{\tau_1}^{(k)}\in {\mathscr R}_{k} }A_{r_{B_j}}\left( P_{B_{\tau_1}^{(k)}}
 \left| (I-e^{-r_{B_j}^mL})^Ke^{itL} \phi_{k}({L})
  \left(\sum_{\tau: \,  d( x^{(k)}_{\tau}, \, x ) \geq 10\cdot 2^{(m-1)k/m}(1+|t|)}
  \left[ \varphi_{k}({L})f\right]^{B^{(k)}_{\tau}}\right)\right|^2
 \right)(x).
\end{eqnarray*}
Applying \eqref{e3rr}, we  see that the $L^{p/2}$-norm of $D_1(x)$
is dominated by a constant times
\begin{eqnarray*}
 \sum_{B_{\tau_1}^{(k)}\in  {\mathscr R}_{k} } \sup_{y\in B_{\tau_1}^{(k)}} [V(y, r_{B_j})]^{-(1-{2\over p})} \left\| P_{B_{\tau_1}^{(k)}}
  (I-e^{-r_{B_j}^mL})^Ke^{itL} \phi_{k}({L})
  \left(\sum_{\tau: \,  d( x^{(k)}_{\tau},\, x_{\tau_1}^{(k)} ) \geq 10\cdot 2^{(m-1)k/m}(1+|t|)}
  \left[ \varphi_{k}({L})f\right]^{B^{(k)}_{\tau}} \right)
 \right\|^2_{2}.
 \end{eqnarray*}
Observe that for every $B_{\tau_1}^{(k)}\in {\mathscr R}_{k},$

\smallskip

 $\bullet$   If $y\in B_{\tau_1}^{(k)},$ then
  \begin{eqnarray*}
 \left({\mu(B_{\tau}^{(k)})\over V(y, r_{B_j})}\right)&=&
 \left({\mu(B_{\tau}^{(k)})\over {\mu(B_{\tau_1}^{(k)})}}\right) \times  \left({\mu(B_{\tau_1}^{(k)})\over V(y, r_{B_j})}\right)
  \leq  C (1+|t|)^n 2^{n[k(m-1) -j]/m} \left(1+{d(B_{\tau_1}^{(k)}, B_{\tau}^{(k)})\over 2^{(m-1)k/m}(1+|t|)}\right )^{D}.
  \end{eqnarray*}

  \smallskip

 $\bullet$  A simple calculation shows that
   \begin{eqnarray*}
 \sharp \left\{ \tau:    2^{{(m-1)k\over m} +u}(1+|t|)\leq  d(B_{\tau_1}^{(k)}, B_{\tau}^{(k)}) \leq 2^{{(m-1)k\over m} +u+1}(1+|t|)
  \right\}
 \le   C 2^{u (D+n)}
 \end{eqnarray*}
 and so
  \begin{eqnarray}\label{pengadd}
 &&\hspace{-1cm}\sum_{\tau: d(B_{\tau_1}^{(k)}, B_{\tau}^{(k)})> 10\cdot 2^{(m-1)k/m}(1+|t|) }
 \left(1+{d(B_{\tau_1}^{(k)}, \, B_{\tau}^{(k)})\over 2^{(m-1)k/m}(1+|t|)}\right )^{-M}\nonumber\\
 &\leq&
 \sum_{u=2}^{\infty}\sum_{\tau:    2^{{(m-1)k\over m} +u}(1+|t|)\leq  d(B_{\tau_1}^{(k)}, B_{\tau}^{(k)})
 \leq 2^{{(m-1)k\over m} +u+1}(1+|t|)} 2^{-uM}
 \leq C \sum_{u=2}^{\infty} 2^{-u(M-(D+n))}\leq C
  \end{eqnarray}
  for $M> D+n$.
Since the function $ \big[ \varphi_{k}({L})f\big]^{B^{(k)}_{\tau}}$
   is supported on the ball  $B^{(k)}_{\tau}$, we apply Proposition~\ref{prop3.3} with $M>D+n$ and the H\"older inequality
   to see that $\|D_1\|_{p/2}$ is controlled by a constant multiple of
 \begin{eqnarray*}
   (1+|t|)^{n(1-{2\over p})} 2^{{n\over m}[k(m-1)-j](1-{2\over p})} \sum_{B_{\tau_1}^{(k)}\in {\mathscr R}_{k}}
 \sum_{\tau: \,  d( x^{(k)}_{\tau}, \, x_{\tau_1}^{(k)} ) \geq 10\cdot 2^{(m-1)k/m}(1+|t|)}
\left(1+{d(B_{\tau_1}^{(k)}, B_{\tau}^{(k)})\over 2^{(m-1)k/m}(1+|t|)}\right )^{-M}
\left\|
  \left[ \varphi_{k}({L})f\right]^{B^{(k)}_{\tau}}
 \right\|_{p}^p.
 \end{eqnarray*}
Changing the order of the summation for $\tau_1$ and $\tau$ and by \eqref{pengadd}, we obtain
  \begin{eqnarray*}
\|D_1\|_{p/2}
 &\leq& C    (1+|t|)^{n(1-{2\over p})} 2^{{n\over m}[k(m-1)-j](1-{2\over p})} \left\|
 \varphi_{k}({L})f
 \right\|_{p}^p.
 \end{eqnarray*}
For the term $D_2,$
we follow the  
similar approach as above in $D_1(x)$ to show that for every $\tau$ with
$\, 6\cdot 2^{(m-1)k/m}(1+|t|) \leq  d( x^{(k)}_{\tau}, x) \leq 10 \cdot 2^{(m-1)k/m}(1+|t|),$
 \begin{eqnarray*}
 \left\|A_{r_{B_j}}\left(
 \left|(I-e^{-r_{B_j}^mL})^Ke^{itL} \phi_{k}({L})
  \left[ \varphi_{k}({L})f\right]^{B^{(k)}_{\tau}}\right|^2 \right)\right\|_{p/2}
  \leq(1+|t|)^{n(1-{2\over p})} 2^{{n\over m}[k(m-1)-j](1-{2\over p})} \left\|
 \varphi_{k}({L})f
 \right\|_{p}^p,
\end{eqnarray*}
and so by \eqref{ll} in {\it Case 3}, we have that
$\|D_2\|_{p/2}\leq (1+|t|)^{n(1-{2\over p})} 2^{{n\over m}[k(m-1)-j](1-{2\over p})} \left\|
 \varphi_{k}({L})f
 \right\|_{p}^p.
$
This finishes the proof  of \eqref{ecc} and  thereby  \eqref{evv} for the term $III_{22}$ and concludes
 that
$$
\|III_2\|_{p}\leq  C(1+|t|)^{n({1\over 2}-{1\over p})}\left(\sum\limits_{k\geq 1} \|\varphi_k({L})  f \|_p^p\right)^{1/p}.
$$
Combining the estimates of $III_1(x)$ and $III_2(x)$, we obtain the estimate for $III(x)$ as desired.

Finally, we combine   estimates of $I, II $ and $ III $  to obtain the estimate \eqref{e3.7}, and complete
the proof of
Theorem~\ref{th1.1}.

\begin{proof}[Proof of Corollary~\ref{cor3.3}]
 The proof of Corollary~\ref{cor3.3}   can be obtained by making a minor modifications with
 \cite[Theorem 7.12]{O}, and we skip it here.
\end{proof}

 \medskip
We mention  that our Theorem~\ref{th1.1} can also apply to prove existence of solution (in $L^p$ spaces)
 to the Schr\"odinger equation with initial data $f$
in the domain of some power of the operator $L$. It can also be formulated in terms of generation of $C$-regularized groups.
We will not develop this here, we refer the reader to de Laubenfels \cite{DeL} and
 Ouhabaz's monograph \cite[Chapter  7]{O}.

\medskip

\section{An application to Riesz means of the solutions of the Schr\"odinger equations}
\setcounter{equation}{0}

The aim of this section is to prove Theorem~\ref{th1.2}.
Recall that when $L$ is the Laplacian on the Euclidean spaces ${\mathbb R}^n$,
the  Riesz mean  $I_s(t)(\Delta)$  in \eqref{RieszMean} was studied  by Sj\"ostrand \cite{Sj} . It was shown  that
$I_s(t)(\Delta)$ is uniformly bounded in $t\in{\mathbb R}\backslash\{0\}$ for
$s> n|{1/2}-{1/p}|$, and they are unbounded for  $s<n|{1/2}-{1/p}|.$ The result was generalized to Lie groups and Riemannian
manifolds by Lohoue\cite{Lo} and  by Alexopoulos \cite{A}.
 In the abstract setting
of operators on metric measure  space,   this result was extended by Carron, Coulhon and Ouhabaz \cite{CCO}
for operators with the Gaussian upper bounds, and by
Blunck \cite{Bl} for generalized Gaussian estimates for the operators.
More precisely,  the work of Blunck  \cite[Proposition A]{Bl}  shows that  under the assumption of
generalized Gaussian estimate  \eqref{GGE}  for some
  $1\leq p_0< 2$,  then the Riesz means operator $I_s(t)(L)$ is bounded on $L^p(X)$ uniformly for all $t\in{\mathbb R}\backslash \{0\},$
  $p\in (p_0, p'_0)$ and $s>  n|{1/ 2}-{1/ p}|$.
  To prove  the endpoint estimate for $s=  n|{1/ 2}-{1/ p}|,$
	we need following result.

	\begin{thm}\label{th6.1}
Suppose  that $(X, d, \mu)$ is  a  space of homogeneous type  with a dimension $n$.  Suppose that $L$
satisfies  \eqref{GGE}  for some
  $1\leq p_0< 2$.
Then for every $p\in (p_0, p'_0)$, there exists a  constant $C=C(n,p)>0$  such that for all $t\in {\mathbb R}\backslash \{0\},$
\begin{eqnarray} \label{e6.5}
 \left\| (I+|t|L)^{-s }e^{itL} f\right\|_{p} \leq C  \|f\|_{p},   \ \
  \ s\geq n\big|{1\over  2}-{1\over  p}\big|.
\end{eqnarray}

As a consequence,  this  estimate \eqref{e6.5} holds for all $1<p<\infty$ when the heat kernel of $L$ satisfies a  Gaussian upper bound
\eqref{GE}.
\end{thm}

\begin{proof}
We prove this theorem by following the approach  in the proof of  Theorem~\ref{th1.1}
 by using  Proposition~\ref{prop13.33} instead of Proposition~\ref{prop3.3}. For the details, we leave to the reader.
\end{proof}

 \medskip

\begin{proof}[Proof of Theorem~\ref{th1.2}] The proof of Theorem~\ref{th1.2} is inspired by the idea of \cite{Sj}.
Take  a function $\Phi\in C^\infty(\R)$ such that $\Phi(t)=0$ if $t<1/2$ and $\Phi(t)=1$ if $t>1$. Define function $F$ by
$$
F(u)=I_s(1)(u)-C_s\Phi(u)u^{-s}e^{-i u},
$$
where $C_s$ is defined by
$$
s\int_{-\infty}^1 (1-\lambda)^{s-1}e^{i\lambda u}d\lambda=C_s u^{-s}e^{iu}, \quad u>0.
$$
It is seen that for $0<u\leq 1$ and $k\in \mathbb{N}$,
$$
\frac{d^k}{d u^k}F(u)\leq C,
$$
and   for $u>1$ and $k\in \mathbb{N},$
$$
\frac{d^k}{d u^k}F(u)\leq Cu^{-k}.
$$
See   \cite[Lemma 2.1]{Sj}.
Hence,  for every $\beta>(n+1)/2$  we have that $
\sup_{R>0}\|\eta\delta_R F\|_{C^\beta}\leq C,
 $
 and so
 $
\sup_{R>0}\|\eta\delta_R F(t\cdot)\|_{C^\beta}\leq C
 $
with a constant $C>0$ independent of $t>0$. Then we apply (a) of Proposition~\ref{prop2.5}
to know that $F(tL)$ is bounded on $L^p(X)$ for all $p_0<p<p_0'$. Notice that for every $t>0,$
\begin{eqnarray} \label{vvc}
F(tL)=I_s(t)(L)-C_s\Phi(tL)(tL)^{-s}e^{-i tL}.
\end{eqnarray}
This yields  that for every $t>0,$
\begin{eqnarray} \label{vvcc}
\|I_s(t)(L)\|_{p\to p}&\leq& \|F(tL)\|_{p\to p}+C\|\Phi(tL)(tL)^{-s}e^{-itL}\|_{p\to p}\nonumber\\
&\leq& C+C\|\Phi(tL)(tL)^{-s}(1+tL)^s\|_{p\to p}\|(1+tL)^{-s} e^{-itL}\|_{p\to p}.
\end{eqnarray}
Applying (a) of Proposition~\ref{prop2.5} again, we have that $ \|\Phi(tL)(tL)^{-s}(1+tL)^s\|_{p\to p}  \leq C. $
This, in combination with  \eqref{e6.5} in Theorem~\ref{th6.1},  implies  $
\|I_s(t)(L)\|_{p\to p}\leq C
$
for $t>0$.

Since $I_{s}(t)(L)={\overline I}_{s}(-t)(L)$  for $t<0$, we have  that $\|I_s(t)(L)\|_{p\to p}\leq C
$
for $t<0$.
The proof of Theorem~\ref{th1.2} is complete.
\end{proof}

\medskip

 \bigskip
\noindent
{\bf Acknowledgements}:
The authors would like to thank the  referee for  helpful comments and suggestions.
P. Chen was supported by Guangdong Natural Science Foundation, Grant No.
2016A030313351.
 X.T. Duong was supported by  the Australian Research Council (ARC) through the research
grant DP190100970.
J.  Li was supported by the Australian Research Council (ARC) through the
research grant DP170101060 and by Macquarie University Research Seeding Grant.
L. Yan was supported by the NNSF of China, Grant
No. 11521101 and 11871480,  Guangdong
Special Support Program, and by the Australian Research Council (ARC) through the research
grant DP190100970. The authors  thank T.A. Bui, Z.H. Guo, E.M. Ouhabaz,  A. Sikora and X.H. Yao for useful discussions.


\vskip 1cm


 \begin{thebibliography}{99999}

  \bibitem  {A} G. Alexopoulos, Oscillating multipliers on Lie groups and Riemannian manifolds. {\it Tohoku Math. J.}
  {\bf 46} (1994), 457-468.


\bibitem{ACDH} P. Auscher, T. Coulhon, X.T. Duong and S. Hofmann,
Riesz transform on manifolds and heat kernel regularity.
{\it Ann. Sci. \'Ecole Norm. Sup.} {\bf 37}  (2004), 911--957.

\bibitem{BE} M. Balabane and H.A. Emamirad, $L^p$ estimates for Schr\"odinger evolution equations. {\it Trans.
Amer. Math. Soc.} {\bf  291} (1985), 357--373.



\bibitem{Bl} S. Blunck, Generalized Gaussian estimates and Riesz means of Schr\"odinger groups.
{\it J. Aust. Math. Soc.} {\bf  82} (2007), 149--162.

\bibitem{Blu}
S. Blunck, A H\"ormander-type spectral multiplier theorem for operators without heat kernel.
{\it Ann. Sc. Norm. Super. Pisa Cl. Sci. (5)} {\bf 2} (2003), no. 3,  449--459.


\bibitem{BK2}  S. Blunck and P.C. Kunstmann,  Calder\'on-Zygmund theory for non-integral operators and the $H^\infty$ functional calculus.
{\it Rev. Mat. Iberoam.} {\bf 19}   (2003),     919--942.

  \bibitem{Br} P. Brenner, The Cauchy problem for systems in $L_p$ and $L_{p, \alpha}$.
 {\it Ark. Mat.} {\bf 2} (1973), 75--101.





  \bibitem{BDN} T.A. Bui, P. D' Ancona and F. Nicola,  Sharp $L^p$ estimates for Schr\"odinger groups
  on space of homogeneous type. {\it Rev. Mat. Iberoam.} {\bf 36} (2020), no. 2, 455--484.





\bibitem{CCO} G. Carron, T.  Coulhon and E.M. Ouhabaz,  Gaussian estimates and $L^p$-boundedness of Riesz means.
{\it J. Evol. Equ.} {\bf 2} (2002), 299--317.



 \bibitem{COSY} P. Chen, E.M. Ouhabaz, A.  Sikora and L.X. Yan,
Restriction estimates, sharp spectral multipliers and endpoint estimates
for Bochner-Riesz means. {\it J. Anal. Math.} {\bf  129}  (2016), 219--283.




\bibitem {CW} R. Coifman and G. Weiss,  Analyse harmonique
non-commutative sur certains espaces homog\`enes. {\it  Lecture Notes in Math.} {\bf 242}. Springer, Berlin-New York, 1971.





  \bibitem{DN} P. D' Ancona and F. Nicola,  Sharp $L^p$ estimates for Schr\"odinger groups.
  {\it  Rev. Mat. Iberoam.}  {\bf 32} (2016),1019--1038.


 \bibitem   {D} E.B. Davies, {\it Heat kernels and spectral theory},
 Cambridge Univ. Press, 1989.


  \bibitem{D2} E.B. Davies, Limits on $L^p$ regularity of self-adjoint elliptic operators.
  {\it J. Differential Equations} {\bf 135} (1997), 83--102.

  \bibitem{DeL} R. deLaubenfels, Existence families, functional calculi and evolution equations.
{\it Lecture Notes in Mathematics}, {\bf 1570}. Springer-Verlag, Berlin, 1994.




\bibitem{DM} X.T. Duong and A. McIntosh,   Singular integral operators with
non-smooth kernels on irregular domains. {\it Rev. Mat. Iberoam.} {\bf 15}
(1999), no. 2, 233--265.


  \bibitem {DOS}  X.T. Duong, E.M. Ouhabaz  and A. Sikora,
 Plancherel-type estimates and sharp spectral multipliers.
 {\it J. Funct. Anal.}, {\bf 196} (2002),  443--485.

 \bibitem {DY} X.T. Duong and L.X. Yan, New function spaces of BMO type,
John-Nirenberg inequality, interpolation and applications. {\it Comm.
Pure Appl. Math.} {\bf 58} (2005), 1375--1420.


\bibitem{FS} C. Fefferman and E.M. Stein, $H^p$ spaces of
 several variables.  {\it Acta
Math.} {\bf 129} (1972), 137--195.


\bibitem{Gi} A. Grigor'yan,  Heat kernel and analysis on manifolds. {\it AMS/IP Studies
in Advanced Mathematics},  {\bf 47}. American Mathematical Society, Providence, RI; International Press, Boston, MA, 2009.


\bibitem{H} M. Hieber,  Integrated semigroups and differential operators on $L^p$ spaces. {\it  Math. Ann.} {\bf  291} (1991),   1--16.


 \bibitem {H1} L. H\"ormander,  Estimates for translation invariant
operators in $L^p$ spaces.   {\it Acta Math.}  {\bf 104} (1960), 93--140.



 \bibitem{JN} A. Jensen and S.  Nakamura,  Mapping properties of functions of Schr\"odinger operators between
$L^p$-spaces and Besov spaces. In {\it Spectral and scattering theory and applications},
187--209. Adv. Stud. Pure Math. {\bf 23}, Math. Soc. Japan, Tokyo, 1994.


\bibitem{JN2} A. Jensen and S.  Nakamura, $L^p$-mapping properties of functions of Schr\"odinger operators
 and their applications to scattering theory. {\it J. Math. Soc. Japan} {\bf 47} (1995),  253--273.

 \bibitem{KU} P. Kunstmann and M. Uhl,   Spectral multiplier theorems of
 H\"ormander type on Hardy and Lebesgue spaces. {\it J. Operator Theory} {\bf 73 }
 (2015),  27--69.

 \bibitem{La} E. Lanconelli, Valutazioni in $L^p(\RN)$ della soluzione del problema di Cauchy per l'equazione di Schr\"odinger.
 {\it Boll. Un. Mat. Ital. (4)} {\bf 1} (1968), 591--607.


 \bibitem{LRS} S. Lee, K. Rogers and A. Seeger, Improved bounds for Stein's square functions.
 {\it Proc. Lond. Math. Soc.} {\bf 104}  (2012),  1198-1234.



  \bibitem {LSV} V. Liskevich, Z. Sobol and H. Vogt, On the $L^p$ theory
  of $C^0$-semigroups associated with second-order elliptic operators {\rm II}.
  {\it J. Funct. Anal.} {\bf 193} (2002),   55--76.

 \bibitem{Lo} N. Lohou\'e, Estimations des sommes de Riesz d'op\'erateurs de Schr\"odinger sur les vari\'et\'es
 riemanniennes et les groupes de Lie. {\it C.R.A.S. Paris}. {\bf 315} (1992), 13-18.



  \bibitem{Ma} Martell, J.M.  Sharp maximal functions associated with
 approximations of the identity in spaces of
 homogeneous type and applications.  {\it Studia Math.}  {\bf
 161} (2004), no. 2, 113-145.

 \bibitem  {Mc} A. McIntosh,  Operators which have an $H_\infty$ functional
calculus, {\it  Miniconference on operator theory and partial differential
equations (North Ryde, 1986)}, 210-231, Proceedings of the Centre  for
Mathematical Analysis, Australian National University, {\bf 14}.
 Australian National University, Canberra, 1986.




 \bibitem{Mi1} A. Miyachi, On some Fourier multipliers for $H^p(R^n)$. {\it J. Fac. Sci. Univ. Tokyo Sect. IA Math.}
 {\bf  27} (1980), 157-179.



    \bibitem{Mi} A. Miyachi, On some singular Fourier multipliers. {\it J. Fac. Sci. Univ. Tokyo Sect. IA Math.} {\bf 28} (1981), 267-315.


\bibitem {O} E.M. Ouhabaz, {\it Analysis of Heat Equations on
Domains}.  London Mathematical Society Monographs Series, {\bf 31}.
Princeton University Press, Princeton, NJ, 2005.



  \bibitem{RS}  K. Rogers and A. Seeger,
  Endpoint maximal and smoothing estimates for
  Schr\"odinger equations. {\it J. Reine Angew. Math.} {\bf 640} (2010), 47--66.


 \bibitem{ScV} G. Schreieck and J. Voigt, Stability of the $L_p$-spectrum of generalized Schr\"odinger
 operators with form small negative part of the potential. In {\it Function Analysis (Essen, 1991)}, 95--105.
 Lecture Notes in Pure and Appl. Math., {\bf 150}. Dekker, New York, 1994.



 \bibitem{SYY} A. Sikora, L.X. Yan and X.H. Yao, Sharp spectral multipliers for
operators satisfying generalized  Gaussian
estimates.   {\it J. Funct. Anal.} {\bf 266} (2014),   368--409.


\bibitem{Si} B. Simon, Schr\"odinger semigroup. {\it Bull. Amer. Math. Soc.} {\bf 7} (1982), 447-526.

\bibitem{Sj} S. Sj\"ostrand, On the Riesz means of the solutions of the Schr\"odinger equation. {\it Ann. Scuola Norm. Sup. Pisa.}
{\bf 24} (1970), 331-348.


  \bibitem   {St2} E.M. Stein,   {\it Singular integral and
differentiability properties of  functions},   Princeton Univ. Press,
{\bf  30}, (1970).



 \bibitem{Ta} M. Taylor, $L^p$ estimates on functions of the Laplace operator. {\it Duke Math. J.} {\bf 58} (1989), 773-793.


\bibitem{Th} S. Thangavelu, Multipliers for Hermite expansions. {\it Rev. Mat. Iberoam.} {\bf 3}
(1987), no. 1, 1-24.

\bibitem{TSC} N.Th. Varopoulos, L. Saloff-Coste and T. Coulhon, {\it Analysis and Geometry on Groups},
Cambridge University Press, Cambridge, 1992.

\end{thebibliography}
 \end{document}